\documentclass{article}

\usepackage[noadjust]{cite}

\usepackage{amsmath,amssymb,amsfonts, amsthm}
\usepackage{mathtools}
\usepackage[margin=1.1in]{geometry}
\usepackage{booktabs}
\usepackage{multirow}
\usepackage{tabularx}

\usepackage{enumitem}
\setlist{topsep=1pt,itemsep=0pt,labelwidth=8pt,labelsep=5pt,leftmargin=16pt}
\newcommand{\AlgFullname}{Zeroth-Order Incremental Variance Reduction}
\newcommand{\AlgAbbrv}{ZIVR}

\usepackage{graphicx}
\usepackage{textcomp}
\usepackage{xcolor}

\usepackage[ruled, norelsize, algo2e]{algorithm2e}
\let\oldnl\nl
\newcommand{\nonl}{\renewcommand{\nl}{\let\nl\oldnl}}
\usepackage{setspace}

\usepackage{comment}

\DeclareMathOperator*{\argmin}{arg\,min}

\newcommand{\defeq}{\coloneqq}

\DeclarePairedDelimiterX\Econd[2]{[}{]}{#1 \mkern2mu\delimsize\vert\mkern2mu #2}

\newtheorem{lemma}{Lemma}
\newtheorem{theorem}{Theorem}
\newtheorem{corollary}{Corollary}

\theoremstyle{definition}
\newtheorem{assumption}{Assumption}

\theoremstyle{remark}
\newtheorem{remark}{Remark}

\newlength\myindent
\allowdisplaybreaks

\title{\LARGE \bf
Distributed Zeroth-Order Composite Finite-Sum Optimization via Incremental Variance Reduction
}

\author{Silan Zhang and Yujie Tang
\thanks{This work was supported by the National Natural Science Foundation of China through grant 72431001.
S. Zhang and Y. Tang are with the Department of Control Science and Systems Engineering, Peking University, Beijing, China. 
        Email: {\tt\small zhangsilan@stu.pku.edu.cn}, {\tt\small yujietang@pku.edu.cn}}%
}

\date{}

\begin{document}

\maketitle

\begin{abstract}
We study decentralized composite finite-sum optimization over a fixed connected undirected network,
where agents access local component functions through function evaluations and share a possibly nonsmooth convex regularizer.
In this setting, standard two-point gradient estimators can retain nonvanishing variance at a solution, complicating their integration with gradient tracking.
We propose Distributed Zeroth-Order Incremental Variance Reduction (Distributed ZIVR), which combines incrementally updated component-gradient estimates with a proximal gradient-tracking scheme.
The method admits an implementation using a constant number of function evaluations per agent at every iteration, without intermittent full-component or full-coordinate refreshes.
For convex smooth component functions and strongly convex local averages, we establish linear convergence up to a controllable finite-difference error.
With the finite-difference radius chosen for the target accuracy, Chebyshev-accelerated mixing yields a per-agent oracle complexity of $\widetilde{\mathcal{O}}(d(m+\kappa)\log(1/\epsilon))$ for expected squared solution error $\epsilon$,
where $d$ is the dimension, $m$ is the number of components per agent, and $\kappa$ is the condition number.
This matches the centralized ZIVR oracle bound, while the communication bound has only an inverse square-root dependence on the network spectral gap.
Experiments on logistic and robust regression demonstrate rapid reduction of solution and consensus errors.
\end{abstract}

\section{Introduction}
\label{section:introduction}

We consider a network of $n$ agents that cooperatively solves the composite optimization problem
\begin{equation}\label{eq:intro_formulation}
    \min_{x\in\mathbb{R}^d}
    \ h(x)\coloneqq \frac{1}{n}\sum_{i=1}^{n} \frac{1}{m_{i}} \sum_{j=1}^{m_{i}} f_{i, j}(x)+\psi(x),
\end{equation}
where $x\in\mathbb{R}^d$ is the common decision variable. Each $f_{i, j}$ is a smooth local objective known only to agent $i$, and $\psi$ is a shared, possibly nonsmooth, convex regularizer.
Agent $i$ may only query the zeroth-order information of the $m_{i}$ functions $f_{i, j}, j \in \{1, 2, \cdots, m_{i}\}$ local to itself, and exchange information with their neighbors.

Distributed optimization finds applications in large-scale machine learning, electric power systems, and distributed sensor fusion~\cite{nedic2020machine, molzahn2017survey,xiao2005sensorfusion}.
In the smooth case $\psi\equiv0$, decentralized gradient descent (DGD) provides a direct way to combine optimization and consensus~\cite{nedic2009distributed}.
However, it is known that with a fixed stepsize, DGD generally converges only to a neighborhood of the solution,
whereas exact convergence requires a diminishing stepsize~\cite{yuan2016convergence}, resulting in suboptimal convergence rate.

The cause of this problem is local heterogeneity. At a global minimizer $x^\star$, the average gradient vanishes, but the individual gradients need not.
As a result, with fixed stepsize local copies of the decision variable will not converge to $x^{\star}$.
Gradient-tracking methods correct this mismatch by assigning each agent a tracking gradient that recursively estimates the average gradient.
The tracking recursion combines neighbor communication with local gradient changes to follow the evolving network-average gradient.
This mechanism permits exact convergence with a fixed stepsize and yields linear convergence for smooth strongly convex objectives under standard network conditions~\cite{qu2018harnessing,nedic2017achieving, pu2021pushpull}.
EXTRA and exact diffusion also achieve exact convergence with a fixed
stepsize through correction steps~\cite{shi2015extra,yuan2019exactdiffusion}.

The aforementioned optimization algorithms rely on first-order information. However, in some practical scenarios such as parameter tuning for machine learning \cite{snoek_practical_2012},
black-box adversarial attacks on deep neural networks \cite{chen_zoo_2017} and model-free control \cite{ren_lqr_2021},
explicit gradient information of the objective function can be very difficult or even impossible to compute, and only function evaluations are accessible.
Various zeroth-order (ZO) optimization methods have been proposed to address this challenge.
In this paper, we particularly focus on algorithms based on zeroth-order gradient estimators, initiated by \cite{nemirovskij1983problem}.
In these algorithms, a ZO gradient estimator is constructed based on function values acquired at a few randomly explored points, then combined with gradient-based optimization frameworks.
In the ideal case, these algorithms can match the complexity of their first-order counterparts up to an \(\mathcal O(d)\) factor.
Several zeroth-order algorithms for distributed optimization also achieve oracle-complexity bounds with the same dimension-dependent overhead \cite{tang2021distributed}.

However, as \cite{tang2021distributed} shows, 2-point gradient estimator and GT do not reconcile with each other well.
The authors attribute this to the non-vanishing variance of the 2-point gradient estimator.
While first-order GT methods successfully address the oscillation of local copy $x_{i}$ around the global optimal point $x^{\star}$ due to heterogeneity of the component functions,
the variance of the 2-point gradient estimator brings back this oscillating behavior.
Therefore convergence result is only established under the DGD framework, with an oracle complexity of $\mathcal O(d/\epsilon)$ in the smooth and gradient dominated case, where $\epsilon$ is the optimality gap.
On the other hand, the $2d$-point full-coordinate gradient estimator is variance-free and combines well with GT methods, leading to linear convergence.
However, it requires evaluating $2d$ points at each iteration, which can be impractical for problems of large dimensions.
Recently \cite{mu2024variance} proposes to address this trade-off by combining variance-reduced zeroth-order estimator with gradient-tracking methods,
achieving a complexity of $\mathcal O(d\kappa/(1-\rho^2)^2 \log(1/\epsilon))$, improving GT–2d’s oracle bound by a factor of $\kappa^{1/3}$
while requiring only a constant number of oracle calls per iteration on average. However, this algorithm still requires occasional full-coordinate refreshes.

\paragraph{Challenges of Our Settings}
Our goal here is to successfully develop zeroth-order algorithms whose convergence rate matches first-order variance-reduction algorithms for problem \eqref{eq:intro_formulation} while avoiding the need of batch sampling.
As is pointed above, the main challenge lies in the non-vanishing variance of the 2-point gradient estimator, and this problem is especially pronounced for \eqref{eq:intro_formulation}.
\begin{itemize}
    \item The heterogeneity of local objectives causes a mismatch between local and global optimal points,
    which means global optimal point is not necessary a stationary point for local objective, resulting in the variance of the local 2-point gradient estimator.
    \item With a nonsmooth term $\psi$, optimality requires
    $0\in\nabla f(x^\star)+\partial\psi(x^\star)$, but
    $\nabla f(x^\star)$ need not vanish. Thus, directional-estimation
    variance can persist even in the centralized setting.
    \item Finally, in \eqref{eq:intro_formulation}, each local objective function has a finite-sum structure, so random sampling of component functions adds an extra layer of variance.
\end{itemize}

\subsection{Main Contributions}

This work develops a distributed extension of the Zeroth-Order Incremental Variance Reduction (\AlgAbbrv{}) method for the composite finite-sum problem~\eqref{eq:intro_formulation}.
The proposed method combines local incremental zeroth-order variance reduction with a proximal gradient-tracking scheme,
allowing the agents to cooperate using only local function evaluations and neighbor-to-neighbor communication. Specifically,
\begin{enumerate}
    \item Each agent maintains an incrementally updated table of estimates of its local component gradients.
    Fresh two-point directional samples are used both to refresh selected entries of this table and to correct its stored average when constructing the local gradient estimator.
    We embed these local estimators into a primal-dual proximal gradient-tracking recursion, inspired by \cite{yuan2019exactdiffusion},
    which simultaneously handles local heterogeneity, network consensus, and the possibly nonsmooth regularizer $\psi$.
    The resulting framework admits a pure 2-point zeroth-order algorithm that does not require intermittent full-batch samplings.

    \item We establish a convergence theory for the proposed distributed \AlgAbbrv{} method over connected undirected networks when the component functions are convex and smooth and each local average is strongly convex.
    Our analysis uses an operator-splitting representation and establishes a concrete oracle complexity result that matches first-order algorithms up to an $\mathcal O(d)$ factor.
    With Chebyshev-accelerated mixing, the pure constant-batch implementation matches the centralized zeroth-order oracle complexity,
    while its communication complexity exhibits only a square-root dependence on the network condition number.
\end{enumerate}

\subsection{Related Works}

\paragraph{Distributed optimization and gradient tracking.}
Consensus subgradient methods and DGD established the basic mix-and-update template for distributed optimization.
Although these methods are simple, constant-stepsize DGD generally converges only to a neighborhood of the solution when the local gradients are heterogeneous \cite{nedic2009distributed,yuan2016convergence}.
EXTRA and exact diffusion remove this bias through correction steps \cite{shi2015extra,yuan2019exactdiffusion}.
Gradient tracking (GT) takes a different route: each agent dynamically estimates the network-average gradient while updating its local decision.
This idea yields exact convergence over time-varying graphs and linear rates under strong convexity \cite{nedic2017achieving,qu2018harnessing}.
It has also been extended to directed networks through Push--Pull \cite{pu2021pushpull} and to composite objectives through proximal primal--dual formulations \cite{xu2021distributed}.

\paragraph{Variance reduction and variance reduction with gradient tracking.}
Variance reduced algorithms were first developed for centralized finite-sum problems. Notably, SVRG uses periodic snapshot gradients as control variates,
while SAGA maintains a component-gradient table \cite{johnson2013accelerating,defazio2014saga}. Prox--SVRG extends SVRG to composite objectives \cite{xiao2014proximal}.
Recursive gradient estimators such as SPIDER and SARAH also yield improved convergence guarantees for nonconvex problems \cite{fang2018spider,nguyen2017nonconvex}.
In decentralized finite-sum optimization, however, an algorithm must control both component-sampling noise and disagreement among agents.
DSA addresses the two effects by inserting a SAGA-type table into EXTRA, whereas GT--VR directly combines local SAGA or SVRG estimators with a gradient tracker \cite{mokhtari2016dsa,xin2020gtvr}.
D--GET and DESTRESS employ SARAH/SPIDER-type gradient estimator with gradient tracking for decentralized nonconvex optimization \cite{sun2020dget,li2022destress}.
while PMGT--SAGA and PMGT--LSVRG combine variance reduction with proximal updates and accelerated multi-consensus.
These two proximal methods attain graph-independent component-gradient complexity but retain network dependence in communication \cite{ye2026pmgt}.

\paragraph{Zeroth-order optimization and variance reduction.}
Randomized gradient estimation is a standard approach to ZO optimization. Complexity guarantees have been established for convex problems,
including accelerated methods for smooth objectives \cite{nesterov_random_2017}, and for finding stationary points of smooth stochastic nonconvex problems \cite{ghadimi2013stochastic}.
For stochastic convex optimization, two-point methods attain rates that match minimax lower bounds up to constants or logarithmic factors \cite{duchi2015optimal},
with the logarithmic gap removed in the nonsmooth Euclidean setting \cite{shamir2017optimal}.
For finite-sum problems, ZO variance reduction must control randomness from both component sampling and directional gradient estimation.
Existing ZO--SVRG and ZO--SPIDER variants address both sources of variance, but require $\mathcal O(md)$ queries during full-batch refreshes \cite{fang2018spider,ji2019improved}.
Proximal extensions of SVRG and SAGA handle nonsmooth nonconvex finite sums,
but either require $\mathcal O(d)$ queries per gradient estimate or rely on bounded-gradient assumptions with limited accuracy guarantees \cite{huang2019faster}.
For strongly convex composite problems, ZPDVR achieves $\mathcal O(d(m+\kappa)\log(1/\epsilon))$ query complexity with a constant expected batch size, although its maximum batch size is $\mathcal O(m)$ \cite{di2024double}.
ZIVR incrementally updates stored component-gradient estimates and admits a two-point implementation with constant maximum batch size,
avoiding intermittent full-component or full-coordinate refreshes \cite{zhang2026zivr}.

\paragraph{Distributed zeroth-order optimization and variance reduction.}
Early distributed ZO methods inserted randomized function-value estimators into consensus or primal--dual schemes \cite{yuan2015randomized,Hajinezhad2017zone}.
The work \cite{tang2021distributed} combines 2-point ZO gradient estimator with DGD and $2d$-point estimator with gradient tracking, with the later achieving linear convergence under PL condition.
Distributed mirror-descent and stochastic primal--dual variants broaden the admissible constraints and noise models, but do not exploit an inner finite sum at each agent \cite{yu2022distributed,yi2022zeroth}.
Their general guarantees are sublinear; Yi et al.\ also establish linear convergence under PL condition and relative bounded second-moment assumptions \cite{yi2022zeroth}.
More recent methods reduce the estimator variance before applying GT. DZOVR uses momentum-based variance reduction for stochastic nonconvex objectives,
while VR--GE achieves an $\mathcal O(1)$ average batch through directional memory but still performs occasional $\mathcal O(d)$ refreshes \cite{chen2025dzovr,mu2024variance}.
Closest to our setting, GT--SAGA--2pVRGE combines component memory, two-point estimation,
and GT to obtain linear convergence with a constant realized batch for smooth local finite sums with a strongly convex global objective \cite{mu2026distributed}.
However, it does not handle a shared nonsmooth regularizer or common constraint, and its complexity result is also worse.

Our method allows a shared proximable $\psi$---and hence a common convex constraint through an indicator function---without intermittent full-component or full-coordinate batches.
The Chebyshev-accelerated specialization in Corollary~\ref{cor:accelerated_complexity} uses $\widetilde{\mathcal{O}}(d(m+\kappa)\log(1/\epsilon))$ oracle calls and $\widetilde{\mathcal{O}}((d(m+\kappa)+\sqrt{d(m+\kappa)/(1-\rho^2)}) \log(1/\epsilon))$ communication rounds,
thereby matching the centralized oracle complexity. Table~\ref{tab:related_work} summarizes the closest strongly convex comparisons.

\begin{table}[!t]
\centering
\caption{Strongly convex computation, communication, and sampling costs.}
\label{tab:related_work}
\scriptsize
\setlength{\tabcolsep}{1.5pt}
\renewcommand{\arraystretch}{1.20}
\begin{tabularx}{\textwidth}{@{}
>{\raggedright\arraybackslash}p{0.14\textwidth}
>{\raggedright\arraybackslash}p{0.115\textwidth}
>{\centering\arraybackslash}X
>{\centering\arraybackslash}X
>{\centering\arraybackslash}p{0.055\textwidth}
>{\centering\arraybackslash}p{0.055\textwidth}
>{\centering\arraybackslash}p{0.07\textwidth}@{}}
\toprule
Setting & Method &
\shortstack{Computation\\(oracle calls)} &
\shortstack{Communication\\rounds} &
\shortstack{Avg.\\batch} & \shortstack{Max.\\batch} & Constr. \\
\midrule
\multirow[t]{2}{0.14\textwidth}{Distributed FO, local finite-sum}
& PMGT--SAGA \cite{ye2026pmgt}
& $\mathcal O((m+\kappa)\log(1/\epsilon))$
& \shortstack{$\mathcal O((m\log m+\kappa\log\kappa)/$\\
  $\sqrt{1-\rho}\,\log(1/\epsilon))$}
& $\mathcal O(1)$ & $\mathcal O(1)$ & Yes \\
\cmidrule(l){2-7}
& PMGT--LSVRG \cite{ye2026pmgt}
& $\mathcal O((m+\kappa)\log(1/\epsilon))$
& \shortstack{$\mathcal O((m\log m+\kappa\log\kappa)/$\\
  $\sqrt{1-\rho}\,\log(1/\epsilon))$}
& $\mathcal O(1)$ & $\mathcal O(m)$ & Yes \\
\midrule
\multirow[t]{3}{0.14\textwidth}{Distributed ZO, deterministic local objectives}
& DGD--2p \cite{tang2021distributed}
& $\mathcal O(d \kappa^{2}/\epsilon)^\dagger$
& $\mathcal O(d\kappa^{2}/\epsilon)^\dagger$
& $\mathcal O(1)$ & $\mathcal O(1)$ & No \\
\cmidrule(l){2-7}
& GT--2d \cite{tang2021distributed}
& \shortstack{$\mathcal O(d\kappa^{4/3}/(1-\rho^2)^2$\\
  $\log(1/\epsilon))$}
& \shortstack{$\mathcal O(\kappa^{4/3}/(1-\rho^2)^2$\\
  $\log(1/\epsilon))$}
& $\mathcal O(d)$ & $\mathcal O(d)$ & No \\
\cmidrule(l){2-7}
& VR--GE \cite{mu2024variance}
& \shortstack{$\mathcal O(d\kappa/(1-\rho^2)^2$\\
  $\log(1/\epsilon))$}
& \shortstack{$\mathcal O(d\kappa/(1-\rho^2)^2$\\
  $\log(1/\epsilon))$}
& $\mathcal O(1)$ & $\mathcal O(d)$ & No \\
\midrule
\multirow[t]{2}{0.14\textwidth}{Centralized ZO finite-sum}
& ZPDVR \cite{di2024double}
& $\mathcal O(d(m+\kappa)\log(1/\epsilon))$
& --
& $\mathcal O(1)$ & $\mathcal O(m)$ & Yes \\
\cmidrule(l){2-7}
& ZIVR \cite{zhang2026zivr}
& $\widetilde{\mathcal{O}}(d(m+\kappa)\log(1/\epsilon))$
& --
& $\mathcal O(1)$ & $\mathcal O(1)$ & Yes \\
\midrule
\multirow[t]{3}{0.14\textwidth}{Distributed ZO, local finite-sum}
& GT--SAGA--2pVRGE \cite{mu2026distributed}
& \shortstack{$\mathcal O(md^2\kappa^2/(1-\rho^2)^2$\\
  $\log(1/\epsilon))$}
& \shortstack{$\mathcal O(md^2\kappa^2/(1-\rho^2)^2$\\
  $\log(1/\epsilon))$}
& $\mathcal O(1)$ & $\mathcal O(1)$ & No \\
\cmidrule(l){2-7}
& \textbf{Distributed \AlgAbbrv{} (ours)}
& \shortstack{$\widetilde{\mathcal{O}}([d(m+\kappa)+$\\
  $(1-\rho^2)^{-1}]\log(1/\epsilon))$}
& \shortstack{$\widetilde{\mathcal{O}}([d(m+\kappa)+$\\
  $(1-\rho^2)^{-1}]\log(1/\epsilon))$}
& $\mathcal O(1)$ & $\mathcal O(1)$ & Yes \\
\cmidrule(l){2-7}
& \textbf{Distributed \AlgAbbrv{}--Cheb (ours)}
& $\widetilde{\mathcal{O}}(d(m+\kappa)\log(1/\epsilon))$
& \shortstack{$\widetilde{\mathcal{O}}([d(m+\kappa)+$\\
  $\sqrt{d(m+\kappa)/(1-\rho^2)}]\log(1/\epsilon))$}
& $\mathcal O(1)$ & $\mathcal O(1)$ & Yes \\
\bottomrule
\end{tabularx}
\vspace{2pt}
\begin{minipage}{\textwidth}
\raggedright\scriptsize
FO and ZO denote first- and zeroth-order oracles.  Here
$\kappa=L/\mu$, $\rho=\|W-n^{-1}\mathbf1\mathbf1^\top\|_2<1$, and
$m_i=m$ in balanced distributed rows.
Distributed oracle counts are per agent, centralized counts are total, and
communication counts synchronous neighbor-exchange rounds rather than
transmitted scalars.  Initial-error factors inside logarithms are suppressed.
For DGD--2p, GT--2d, and VR--GE, one oracle call evaluates $f_i$;
in the finite-sum setting, this costs $m_i$ component-function evaluations.
The dagger marks the source-reported asymptotic DGD--2p rate, which suppresses
graph-dependent stepsize offsets and constants.  Strong convexity implies the
gradient-dominance
condition used by DGD--2p, GT--2d, and VR--GE.  A batch counts fresh component
gradients (FO) or component--direction pairs (ZO), excluding one-time
initialization;   The ZIVR rows use
their pure $R=1$ implementations.  Our unaccelerated and accelerated rates are
from Corollaries~\ref{cor:unaccelerated_complexity} and~\ref{cor:accelerated_complexity}, respectively.
\end{minipage}
\end{table}

\subsection{Notations}
We use $\mathcal O(\cdot)$ for order bounds and $\widetilde{\mathcal{O}}(\cdot)$ when suppressing polylogarithmic factors in problem parameters independent of $\epsilon$; the dependence on $\epsilon$ remains explicit.
For a positive integer $p$, let $[p]\coloneqq\{1,\ldots,p\}$. For a set $\mathcal S$, $|\mathcal S|$ denotes its cardinality.
For $x,y\in\mathbb R^d$, the standard inner product and Euclidean norm are $\langle x,y\rangle=x^\top y$ and $\|x\|=\sqrt{x^\top x}$, respectively.
We write $\mathbb S_d\coloneqq\{x\in\mathbb R^d:\|x\|=1\}$ for the unit sphere, $e_q\in\mathbb R^d$ for the $q$-th standard basis vector ($q\in[d]$), and $\mathbf1$ for the all-ones vector of the appropriate dimension.
For $V,H\in\mathbb R^{r\times s}$, the Frobenius inner product and norm are $\langle V,H\rangle_F=\operatorname{tr}(V^\top H)$ and $\|V\|_F=\sqrt{\langle V,V\rangle_F}$, respectively.
For $i\in[r]$ and $q\in[s]$, $[V]_{i,:}$ and $[V]_{:,q}$ denote the $i$-th row and $q$-th column; $[V]_{i,:}^\top$ expresses the row as a column vector.
For $\mathcal I\subseteq[r]$ and $\mathcal Q\subseteq[s]$, $[V]_{\mathcal I,\mathcal Q}$ denotes the corresponding submatrix.
We write $I_d$ for the $d\times d$ identity matrix. For a symmetric matrix $H\in\mathbb R^{p\times p}$, we order its eigenvalues as $\lambda_1(H)\leq\cdots\leq\lambda_p(H)$.

\section{Problem Formulation}
In this paper, we consider $n\ge2$ agents that cooperatively solve
\begin{equation}\label{eq:formulation}
  \min_{x\in \mathbb{R}^d}\ h(x)\stackrel{\mathrm{def}}{=} f(x)+\psi(x),
\end{equation}
where
\begin{equation*}
  f(x)=\frac{1}{n}\sum_{i=1}^{n} f_i(x),
  \qquad
  f_{i}(x) \coloneqq \frac{1}{m_{i}} \sum_{j=1}^{m_{i}} f_{i,j}(x).
\end{equation*}
Here $d$ and each $m_i$ are positive integers, and $[r]\coloneqq\{1,\ldots,r\}$ for a positive integer $r$.

The agents exchange messages synchronously over a fixed communication graph $\mathcal G=(\mathcal V,\mathcal E)$,
where $\mathcal V=[n]$ is the node set and $\mathcal E$ is the edge set. For each agent $i\in[n]$, define its neighborhood by
\[
    \mathcal N_i\coloneqq
    \{\ell\in[n]:\{i,\ell\}\in\mathcal E\}.
\]

\begin{assumption}[Communication network]\label{assump:communication_network}
The graph $\mathcal G$ is connected and undirected.
Its fixed mixing matrix $W=[w_{i\ell}]\in\mathbb R^{n\times n}$ is symmetric, nonnegative, and row-stochastic, and its smallest eigenvalue is strictly greater than $-1$:
\[
    W=W^\top,\qquad W\mathbf 1_n=\mathbf 1_n,\qquad
    w_{i\ell}\geq 0,\qquad \lambda_{\min}(W)>-1.
\]
For every pair of distinct nodes $i,\ell\in[n]$, the weights respect the graph:
\[
    w_{i\ell}>0\quad\Longleftrightarrow\quad \{i,\ell\}\in\mathcal E.
\]
The condition $\lambda_{\min}(W)>-1$ rules out a periodic mixing matrix; it is, for example, guaranteed when every node has a positive self-weight $w_{ii}>0$.
\end{assumption}

 We also impose the following assumptions on the information that can be accessed by agent $i$ with one oracle call:
\begin{itemize}
\item Given any $\alpha>0$, agent $i$ is able to efficiently evaluate the following proximal operator:
\[
\operatorname{prox}_{\alpha\psi}(x) \defeq
\argmin_{y\in\mathbb{R}^d}\left\{
\frac{1}{2}\|x-y\|_{2}^2 + \alpha\,\psi(y)
\right\}.
\]

\item Agent $i$ may only choose one component function $f_{i,j}$ and obtain the value of $f_{i,j}$ by querying a black-box, or a zeroth-order oracle,
which returns the value $f_{i,j}(x)$ whenever a point $x\in\mathbb{R}^d$ is fed into the black-box as its input. Agent $i$ does not have access to the derivatives of any order for any $f_{i,j}$.
\end{itemize}
Some discussions of these two assumptions are as follows:
\begin{itemize}
\item The first assumption is standard in composite optimization and has been imposed in relevant studies such as~\cite{huang2019faster,pougkakiotis2023zeroth}.
It allows us to incorporate simple non-smoothness into our problem settings and algorithm design. Two representative examples are given below.
  \begin{enumerate}[topsep=1.5ex]
    \item If we take $\psi$ to be the indicator function of a nonempty closed convex set $\mathcal{X}\subseteq\mathbb{R}^d$,
    then \eqref{eq:formulation} is equivalent to the constrained problem $ \min_{x\in\mathcal{X}} f(x) $,
and the proximal operator reduces to projection onto $\mathcal{X}$. 
\item Another practical example is the lasso regression problem~\cite{tibshirani1996lasso}, in which each $f_{i}(x)$ is the loss function corresponding to individual data points,
while $\psi(x) = \lambda \lVert x\rVert_{1}$ for some $\lambda > 0$ is a non-smooth regularization term that enforces sparsity on the solution $x$.
\end{enumerate}

\item The second assumption, as we have discussed in Section~\ref{section:introduction},
is relevant to many problems in areas such as signal processing and machine learning, and necessitates the development of zeroth-order algorithms.
To address the issue of lacking gradient information, different zeroth-order gradient estimators have been proposed. Below we highlight two existing common choices:
  \begin{enumerate}[topsep=1.5ex]
        \item For each $f_{i,j}$, we use the forward 2-point estimator with random perturbation direction $u \in \mathbb{R}^d$ and smoothing radius $\beta > 0$:
        \begin{equation}\label{eq:fm_2p}
            \widehat \nabla_u f_{i,j}(x; \beta) \defeq \frac{f_{i,j}(x+ \beta u) - f_{i,j}(x)}{\beta}u.
        \end{equation}
        This forward-difference convention is used throughout the paper.
        When $\beta$ is sufficiently small, $\widehat \nabla_{u} f_{i,j}(x ; \beta) \approx u u^{T} \nabla f_{i,j}(x)$, i.e., the estimator approximates the directional derivative multiplied by $u$.
        Under the $L$-smoothness condition imposed below, when $u$ is sampled uniformly from $\mathbb{S}_d$, $d\widehat \nabla_u f_{i,j}(x; \beta)$ satisfies
        \begin{gather*}
        \left\|\mathbb{E}_u\!\left[d\widehat\nabla_u f_{i,j}(x; \beta)\right]
        -\nabla f_{i,j}(x)\right\|^2
        \le \frac{dL^2\beta^2}{4},\\
        \operatorname{Var}_u\!\left(d\widehat\nabla_u f_{i,j}(x; \beta)\right)
        \le 2(d-1)\|\nabla f_{i,j}(x)\|_2^2
        +\frac{d^2L^2\beta^2}{2}.
        \end{gather*}
        These bounds follow from Lemma~\ref{lem:forward_difference_error} and the standard moment identities for uniform sphere sampling;
        see also \cite[Lemma~1]{duchi2015optimal} for related bias and second-moment bounds.
        Moreover, as $\beta\to0$, its variance converges to $(d-1)\|\nabla f_{i,j}(x)\|^2$, which need not vanish at a solution of the finite-sum composite problem.
        
    \item Another method to construct zeroth-order estimators is to sample along every coordinate direction, resulting in the following $(d+1)$-point estimator:
        \begin{equation}\label{eq:fm_dp}
            \widehat \nabla f_{i,j}(x; \beta) \defeq \sum_{q=1}^{d} \frac{f_{i,j}(x+ \beta e_q) - f_{i,j}(x)}{\beta}e_q,
        \end{equation}
        where $e_q$ is the $q$-th standard basis vector.
        This deterministic estimator has no variance and produces a fairly accurate estimation of $\nabla f_{i,j}(x)$ when $f_{i,j}$ is smooth and $\beta$ is small enough.
        However, its construction requires $d+1$ function queries, which can be problematic if $d$ is large.
    \end{enumerate}
\end{itemize}

For the convergence analysis, we impose the following assumptions on $f_{i,j}:\mathbb R^d\to\mathbb R$, their local averages $f_i$, and the possibly nonsmooth term $\psi:\mathbb R^d\to\mathbb R\cup\{+\infty\}$.

\begin{assumption}\label{assump:basic_problem}
    \begin{enumerate}
        \item\label{assump:smoothness} Each component function $f_{i,j}$ is convex and $L$-smooth on $\mathbb{R}^d$ for some $L>0$, with
\[
  \|\nabla f_{i,j}(x)-\nabla f_{i,j}(y)\|\leq L\|x-y\|,\qquad\forall x,y\in\mathbb{R}^d.
\]
Consequently, each local average $f_i$ is also $L$-smooth.

\item\label{assump:strong_convex_local} Each local average $f_i$ is $\mu$-strongly convex on $\mathbb{R}^d$ for a common $\mu>0$, i.e., for all $x, y \in \mathbb{R}^d$,
\begin{equation*}
f_i(x) \geq f_i(y) + \langle \nabla f_i(y), x-y \rangle + \frac{\mu}{2}\|x-y\|^2.
\end{equation*}

\item \label{assump:convex_psi} The nonsmooth term $\psi$ is proper, lower semicontinuous, and convex.
Here proper means that its effective domain $\operatorname{dom}\psi\coloneqq\{x\in\mathbb R^d:\psi(x)<+\infty\}$ is nonempty.

\end{enumerate}
\end{assumption}

Under Assumption~\ref{assump:basic_problem}, $h=f+\psi$ is proper, lower semicontinuous, and $\mu$-strongly convex.
In finite dimensions, these properties imply coercivity and existence of a unique minimizer $x^*$.
The assumptions on $\psi$ also ensure that $\operatorname{prox}_{\alpha\psi}$ is single-valued on $\mathbb R^d$ and nonexpansive for every $\alpha>0$.

\section{Our Algorithm}
\label{section:algorithm_design}

The discussion in the previous section highlights a tradeoff in zeroth-order optimization: a few directional queries give an inexpensive gradient estimate whose variance may not vanish near a solution,
whereas estimating every component gradient along all coordinates requires many function evaluations.
To address this tradeoff, we extend the centralized \AlgAbbrv{} construction~\cite{zhang2026zivr} to the distributed setting.
Each agent incrementally collects component-gradient information and uses this memory to reduce the variance of fresh directional samples.
We then combine the local estimators with a gradient-tracking mechanism based on the primal-dual proximal framework of~\cite{xu2021distributed},
so that the agents can cooperatively solve the global problem. The resulting method is summarized in Algorithm~\ref{alg:zivr_distributed}.

\subsection{Construction of the Zeroth-Order Estimator}

We first describe the estimator at a fixed agent $i$. Let $x_i^k\in\mathbb R^d$ be its local solution estimate at iteration $k$, and let $\beta^k>0$ be the finite-difference radius.
Our aim is to approximate $\nabla f_i(x_i^k)=m_i^{-1}\sum_{j=1}^{m_i}\nabla f_{i,j}(x_i^k)$ using a small number of new function evaluations at each iteration.

\paragraph{Component-gradient memory and Jacobian table.}
For each component $j\in[m_i]$, agent $i$ maintains a vector $s_{i,j}^k\in\mathbb R^d$ that estimates $\nabla f_{i,j}(x_i^k)$ using previously queried directional information.
We arrange these vectors as the rows of an estimated Jacobian table,
\[
    \tilde J_i^k
    \coloneqq
    \begin{bmatrix}
        (s_{i,1}^k)^\top\\
        \vdots\\
        (s_{i,m_i}^k)^\top
    \end{bmatrix}
    \in\mathbb R^{m_i\times d}.
\]
To specify what this table should track, for $X=[x_1,\ldots,x_n]^\top\in\mathbb R^{n\times d}$ define the exact table
\[
    J_i(X)
    \coloneqq
    \begin{bmatrix}
        \nabla f_{i,1}(x_i)^\top\\
        \vdots\\
        \nabla f_{i,m_i}(x_i)^\top
    \end{bmatrix}
    \in\mathbb R^{m_i\times d}.
\]
The goal is to keep $\tilde J_i^k$ close to $J_i(X^k)$ by refreshing a sampled part of the table at each iteration.

Specifically, agent $i$ generates a nonempty random set of component-direction pairs $\mathcal S_i^k\subseteq[m_i]\times\mathbb R^d$.
Each pair $(j,u)$ specifies a component $f_{i,j}$ and a unit direction $u$ along which fresh information can be queried. The forward estimator in~\eqref{eq:fm_2p} gives
\[
    \widehat\nabla_u f_{i,j}(x_i^k;\beta^k)
    =\frac{f_{i,j}(x_i^k+\beta^k u)-f_{i,j}(x_i^k)}{\beta^k}u
    \approx uu^\top\nabla f_{i,j}(x_i^k),
\]
where the approximation accounts for the finite-difference error. Thus each pair provides an estimate of the component gradient projected along $u$. Let
\[
    \mathcal I_i^k
    \coloneqq
    \{j\in[m_i]\mid \exists u\ \text{such that}\ (j,u)\in\mathcal S_i^k\}
\]
collect the component indices selected for refresh. Note that here we require the directions associated with any one component to be orthonormal,
so these corrections update distinct orthogonal parts of the stored vector and preserve its projection onto the remaining directions.

At every iteration, agent $i$ refreshes all sampled directions. For $j\in\mathcal I_i^k$, we set
\[
    s_{i,j}^{k+1}
    =s_{i,j}^k
    +\sum_{u:(j,u)\in\mathcal S_i^k}
    \left(
        \widehat\nabla_u f_{i,j}(x_i^k;\beta^k)
        -uu^\top s_{i,j}^k
    \right),
\]
and for $j\notin\mathcal I_i^k$, we set $s_{i,j}^{k+1}=s_{i,j}^k$. The correction subtracts the stored projection along each sampled direction and replaces it with a current finite-difference observation.

\paragraph{Gradient estimator $g_i^k$.}
The average of the stored vectors provides an inexpensive approximation of $\nabla f_i(x_i^k)$, but it may be inaccurate because the table contains older information.
We correct this average using a fresh batch of directional observations. Fix a positive integer batch size $R_i$.
Let $\mathcal R_i^k$ be a multiset of $R_i$ independent component-direction pairs $(j,u)$, where $j$ is drawn uniformly from $[m_i]$ and $u$ independently and uniformly from the unit sphere $\mathbb S_d$.
The batch is sampled independently of the past and of $\mathcal S_i^k$. All sums over $\mathcal R_i^k$ count repeated samples with multiplicity.
Using the table $\tilde J_i^k$ available at the start of the iteration, we form
\[
    g_i^k
    =\frac{1}{m_i}\sum_{j=1}^{m_i}s_{i,j}^k
    +\frac{d}{R_i}
    \sum_{(j,u)\in\mathcal R_i^k}
    \left(
        \widehat\nabla_u f_{i,j}(x_i^k;\beta^k)
        -uu^\top s_{i,j}^k
    \right).
\]
The first term averages all stored component-gradient estimates; the second estimates the correction needed to bring that average up to date.
The factor $d$ compensates for observing only one direction, since $\mathbb E_u[d\,uu^\top]=I_d$.
Together with uniform component sampling, this identity ensures that the stored terms cancel in expectation:
\[
    \mathbb E\!\left[g_i^k\mid X^k,\tilde J_i^k\right]
    =\frac{d}{m_i}\sum_{j=1}^{m_i}
      \mathbb E_u\!\left[\widehat\nabla_u f_{i,j}(x_i^k;\beta^k)\right]
    \approx\nabla f_i(x_i^k).
\]
Consequently, using an older table does not introduce an additional bias into $g_i^k$; the remaining bias comes from the finite differences.

The same correction explains the variance reduction. Ignoring the finite-difference error, independence of the $R_i$ draws gives
\[
    \operatorname{Var}\!\left(g_i^k\mid X^k,\tilde J_i^k\right)
    \approx\frac{d^2}{R_i}\operatorname{Var}_{j,u}\!\left(
        uu^\top\bigl(\nabla f_{i,j}(x_i^k)-s_{i,j}^k\bigr)
    \right),
\]
where the variance on the right is over a single uniformly sampled component $j$ and an independent uniform unit direction $u$, with $X^k$ and $\tilde J_i^k$ held fixed.
Thus the variance is governed by the error in the stored component gradients.
As the table becomes more accurate, the random corrections become smaller, even when the component gradients themselves do not vanish at the solution.

\paragraph{Conditions on table sampling.}
We now formalize the requirement that the refresh should cover all parts of the table over time.
For a stacked table $H=[H_1^\top,\ldots,H_n^\top]^\top$ with $H_i\in\mathbb R^{m_i\times d}$, define the sampling operator $\Pi^k$ by
\begin{equation}\label{eq:sampling_projection_definition}
[(\Pi^kH)_i]_{j,:}
\coloneqq
\left(
\sum_{u:(j,u)\in\mathcal S_i^k}
uu^\top[H_i]_{j,:}^\top
\right)^\top .
\end{equation}
Here $(\Pi^kH)_i$ is the block corresponding to agent $i$. The operator retains exactly the rows and directions selected for refresh. Collect the random choices at iteration $k$ in
\[
    \xi^k\coloneqq
    \{\mathcal S_i^k,\mathcal R_i^k\}_{i=1}^n,
    \qquad k\ge0,
\]
and define the filtration by
\begin{gather*}
    \mathcal F_0\coloneqq
    \sigma\!\left((\zeta_i^0,\tilde J_i^0)_{i=1}^n\right),\\
    \mathcal F_k\coloneqq
    \mathcal F_0\vee\sigma(\xi^0,\ldots,\xi^{k-1}),
    \qquad k\ge1.
\end{gather*}
Here $\sigma(\cdot)$ denotes the generated sigma-field, and $\vee$ denotes the join of sigma-fields.
Thus $X^k$ and $\tilde J^k$ are $\mathcal F_k$-measurable, and conditioning on $\mathcal F_k$ holds the entire history before the current sampling fixed.

\begin{assumption}\label{assump:sampling_operator}
The following conditions hold:
\begin{itemize}
    \item For every agent $i\in[n]$, the random sets $\mathcal S_i^k$, $k=0,1,2,\ldots$, are i.i.d., nonempty almost surely, and each set is independent of $\mathcal F_k$.
    \item For every iteration $k\ge0$, agent $i\in[n]$, and component $j\in[m_i]$, the set of sampled directions $\{u\mid(j,u)\in\mathcal S_i^k\}$ consists of orthonormal vectors.
    \item There is a common $\sigma\in(0,1]$ such that, for every deterministic stacked table $H$ and every $k\ge0$,
    \begin{equation}\label{eq:expected_sampling_projection}
        \mathbb E[\Pi^kH\mid\mathcal F_k]=\sigma H
        \qquad\text{almost surely}.
    \end{equation}
\end{itemize}
\end{assumption}

The orthonormality condition makes $\Pi^k$ an orthogonal projection for every realization: $(\Pi^k)^2=\Pi^k$ and $\|\Pi^kH\|_F^2=\langle\Pi^kH,H\rangle_F$.
The third condition ensures that every table direction is refreshed at the same average rate $\sigma$, determined by the sampled sets $\mathcal S_i^k$.
To see its role, fix the current iterates and ignore the finite-difference error. Then the refresh satisfies the approximate relation
\[
    \mathbb E\!\left[\tilde J_i^{k+1}-J_i(X^k)
    \,\middle|\,\mathcal F_k\right]
    \approx(1-\sigma)\bigl(\tilde J_i^k-J_i(X^k)\bigr).
\]
This suggests that the table can track the current component gradients as the iterates settle, thereby reducing the variance of $g_i^k$.

A concrete construction, following the centralized \AlgAbbrv{} method, is to draw an orthogonal matrix $Q_i^k\in\mathbb R^{d\times d}$,
sample a nonempty set of distinct pairs $(j,q)$ from $[m_i]\times[d]$, and include $(j,Q_i^ke_q)$ in $\mathcal S_i^k$. Distinct columns of $Q_i^k$ ensure orthogonality within each sampled row.
Choosing $Q_i^k=I_d$ gives coordinate directions; drawing $Q_i^k$ uniformly from the orthogonal group gives orthogonal directions whose individual distributions are uniform on the sphere.
The distribution of the selected pairs is chosen so that \eqref{eq:expected_sampling_projection} holds.

 If we use a unit-batch implementation of $R_i=1$ and $|\mathcal S_i^k|=1$,
then the estimator and table refresh each use an independently sampled uniform component and uniform unit direction.
Each iteration therefore requires at most four function evaluations per agent.

\subsection{Gradient Tracking Technique}

Each $g_i^k$ estimates a local gradient, whereas the agents seek a common solution of the global objective.
We therefore combine the estimators with neighbor averaging and a correction that incorporates changes in the local gradient estimates.
The gradient tracking scheme used here is similar to the exact diffusion algorithm introduced in \cite{yuan2019exactdiffusion}.

Agent $i$ stores an auxiliary vector $\zeta_i^k\in\mathbb R^d$ that serves as the input to the proximal operator. At the start of iteration $k$, it computes
\[
    x_i^k=\operatorname{prox}_{\gamma\psi}(\zeta_i^k),
\]
where $\gamma>0$ is the stepsize. This step incorporates the nonsmooth term $\psi$.
After forming $g_i^k$ and refreshing its table as described above, the agent constructs an intermediate vector $\chi_i^k$ to send to its neighbors:
\[
    \chi_i^k=
    \begin{cases}
        x_i^0-\gamma g_i^0, & k=0,\\
        \zeta_i^k+x_i^k-x_i^{k-1}
        -\gamma(g_i^k-g_i^{k-1}), & k\geq1.
    \end{cases}
\]
The first iteration uses a gradient step from the local proximal point. At later iterations, $\chi_i^k$ adds the change in $x_i^k-\gamma g_i^k$ to the previous auxiliary state $\zeta_i^k$.
In particular, the difference $g_i^k-g_i^{k-1}$ carries the new gradient information into the communication step while retaining information from previous exchanges. Agent $i$ then updates
\[
    \zeta_i^{k+1}
    =\frac{1}{2}\chi_i^k
    +\frac{1}{2}\sum_{\ell\in\mathcal N_i\cup\{i\}}
        w_{i\ell}\chi_\ell^k.
\]
This weighted average combines its own intermediate vector with those of its neighbors and supplies the next proximal input.

\begin{algorithm2e}[!tb]
\DontPrintSemicolon
\SetAlgoLined
\LinesNumbered
\SetKwInput{KwParam}{Parameters}
\SetKwInput{KwInit}{Initialize}
\KwParam{Stepsize $\gamma>0$; finite-difference radii $\beta^k>0$ for $k\geq0$; positive integer batch size $R_i$.}
\KwInit{ $\zeta_{i}^{0}\in\mathbb{R}^d$; $s_{i,j}^{0}\in\mathbb{R}^d$ for every $j\in[m_i]$.}
\For{$k=0,1,2,\ldots$}{
  Update the local estimate of the solution:\;
  $\displaystyle x_{i}^{k} = \operatorname{prox}_{\gamma \psi}(\zeta_{i}^{k})$.\;
  
  \vspace{3pt}
  Sample a nonempty set $\mathcal{S}_{i}^{k}\subseteq[m_{i}]\times\mathbb{R}^d$ satisfying Assumption~\ref{assump:sampling_operator}.\;
  
  Sample an indexed batch of $R_{i}$ independent pairs $\mathcal{R}_{i}^{k}\coloneqq \bigl((j_{i, r}^{k},u_{i, r}^{k})\bigr)_{r=1}^{R_{i}}$, where each component index is drawn uniformly from $[m_{i}]$ with replacement, and each direction is drawn uniformly from $\mathbb{S}_{d}$ \;
  \vspace{3pt}
  
  Construct the gradient estimator $g_{i}^{k}$ by:\;
  $\displaystyle g_{i}^{k} 
  = \frac{1}{m_{i}}\sum_{j=1}^{m_i}s_{i,j}^k + \frac{d}{R_{i}}\!\!\sum_{(j, u) \in\mathcal{R}_{i}^{k}}\!
  \bigl(\widehat\nabla_u f_{i,j}(x_{i}^{k};\beta^{k})
    -uu^\top s_{i,j}^{k}
  \bigr).$ \;
  \vspace{3pt}
  
  For each $j\in \mathcal{I}_{i}^{k}$, update the stored vector $s_{i,j}^{k+1}$ by:\;
  \vspace{-5pt}
  \[
    s_{i,j}^{k+1}
    =s_{i,j}^{k}
    +\sum\limits_{u:(j, u) \in \mathcal{S}_{i}^{k}}
    \Bigl(\widehat \nabla_{u} f_{i, j} (x_i^{k};\beta^{k})
    -u u^{\top}s_{i,j}^{k}\Bigr).
  \]\;
  \vspace{-15pt}
  
  For each $j\notin \mathcal{I}_{i}^{k}$, set $s_{i,j}^{k+1}=s_{i,j}^{k}$.\;
  \vspace{3pt}

  \eIf{$k = 0$}{
    Node $i$ computes its initial intermediate state variable $\chi_i^{0}$:\;
    $\displaystyle \chi_i^{0} = x_i^{0} - \gamma g_i^{0}$.\;
  }{
    Node $i$ computes its intermediate state variable $\chi_i^{k}$ using history:\;
    $\displaystyle \chi_i^{k} = \zeta_i^{k} + x_i^{k} - x_i^{k-1} - \gamma \left( g_i^{k} - g_i^{k-1} \right)$.\;
  }
  
  Node $i$ broadcasts $\chi_i^{k}$ to its neighbors and receives their corresponding values $\chi_\ell^{k}$.\;

  Node $i$ updates its $\zeta$-variable for the next iteration (for the accelerated variant, replace this update and the preceding exchange by Algorithm~\ref{alg:chebyshev_consensus}):\;
  $\displaystyle \zeta_i^{k+1} = \frac{1}{2} \chi_i^{k} + \frac{1}{2} \sum_{\ell \in \mathcal{N}_i \cup \{i\}} w_{i\ell} \chi_\ell^{k}$.\;
}
\caption{\AlgFullname{} (\AlgAbbrv{}) at each node $i$}\label{alg:zivr_distributed}
\end{algorithm2e}

\subsection{Chebyshev-Accelerated Consensus}
\label{subsec:chebyshev_consensus}

We use the standard Chebyshev consensus subroutine of~\cite{montijano2013chebyshev} to accelerate the communication step of Algorithm~\ref{alg:zivr_distributed}.
After computing $\chi_i^k$, each agent runs Algorithm~\ref{alg:chebyshev_consensus} to obtain its next proximal input $\zeta_i^{k+1}$ through $K$ neighbor-exchange rounds.
This replaces the single exchange and averaging update at every iteration, including $k=0$.

Let $\rho\coloneqq\|W-\mathbf 1_n\mathbf 1_n^\top/n\|_2<1$ be known to all agents; a common upper bound below one can be used in its place.
In the subroutine, $r$ indexes communication rounds within iteration $k$, and the scalar coefficients $a_r$ are computed locally from $\rho$.

\begin{algorithm2e}[htbp]
\DontPrintSemicolon
\SetAlgoLined
\LinesNumbered
\KwIn{Intermediate state $\chi_i^k$ from Algorithm~\ref{alg:zivr_distributed};
mixing weights $w_{i\ell}$; common spectral bound $0<\rho<1$;
integer number of rounds $K\geq1$.}
\KwOut{Next proximal input $\zeta_i^{k+1}$.}
Set $a_0=1$, $a_1=1/\rho$, and $\chi_i^{k,0}=\chi_i^k$.\;
Exchange $\chi_i^{k,0}$ with neighbors and set
$\displaystyle \chi_i^{k,1}=\sum_{\ell\in\mathcal N_i\cup\{i\}}
 w_{i\ell}\chi_\ell^{k,0}$.\;
\For{$r=1,\ldots,K-1$}{
    Set $a_{r+1}=2a_r/\rho-a_{r-1}$.\;
    Exchange $\chi_i^{k,r}$ with neighbors.\;
    $\displaystyle \chi_i^{k,r+1}
    =\frac{2a_r}{\rho a_{r+1}}
      \sum_{\ell\in\mathcal N_i\cup\{i\}}w_{i\ell}\chi_\ell^{k,r}
      -\frac{a_{r-1}}{a_{r+1}}\chi_i^{k,r-1}$.\;
}
\Return{$\displaystyle \zeta_i^{k+1}=\frac{\chi_i^k+\chi_i^{k,K}}{2}$}.\;
\caption{Chebyshev-accelerated communication step of
Algorithm~\ref{alg:zivr_distributed} at node $i$}
\label{alg:chebyshev_consensus}
\end{algorithm2e}

Each outer iteration uses $K$ communication rounds and no additional function queries. For $K=1$, the subroutine recovers the original update; if $\rho=0$, one ordinary averaging round suffices.
Theorem~\ref{thm:zivr_chebyshev_complexity} specifies $K$ and quantifies the resulting improvement in communication complexity.

\section{Convergence Result}
\label{section:convergence_result}

For both results, the initial primal vectors and table entries satisfy
\begin{equation}\label{eq:initialization_moments}
    \mathbb E\!\left[\sum_{i=1}^n\left(
      \|x_i^0\|^2+\sum_{j=1}^{m_i}\|s_{i,j}^0\|^2
    \right)\right]<\infty.
\end{equation}
Here $x_i^0=\operatorname{prox}_{\gamma\psi}(\zeta_i^0)$, so arbitrary deterministic initialization satisfies this condition.

Define the communication matrices
\begin{equation}\label{eq:communication_matrices}
    A\coloneqq\frac{I_n+W}{2},\qquad
    B\coloneqq I_n-A=\frac{I_n-W}{2},\qquad
    C\coloneqq B^{1/2},
\end{equation}
where $B^{1/2}$ is the symmetric positive-semidefinite square root of $B$. We write $\lambda_2(B)$ for the second-smallest eigenvalue of $B$.

Let $M\coloneqq \sum_{i=1}^{n}m_i$ and $m_{\max}\coloneqq \max_i m_i$. Define the expected number of sampled table directions and the estimator sampling constants by
\begin{equation}\label{eq:sampling_constants}
\begin{gathered}
    \nu\coloneqq
    \mathbb E\!\left[\sum_{i=1}^n|\mathcal S_i^k|
      \,\middle|\,\mathcal F_k\right],\\
    \rho_0 \coloneqq \max_{1\le i\le n}\left(\frac{1}{m_i}+\frac{2d}{m_iR_i}\right),
    \qquad
    \rho_J \coloneqq \max_{1\le i\le n}\frac{2d}{m_iR_i}.
\end{gathered}
\end{equation}
The conditional expectation defining $\nu$ is a deterministic constant independent of $k$ by Assumption~\ref{assump:sampling_operator}.
Since $\operatorname{rank}(\Pi^k)=\sum_i|\mathcal S_i^k|$, taking operator traces in \eqref{eq:expected_sampling_projection} also gives, almost surely, $\nu=\sigma Md$ (and hence $0<\nu\le Md$).
This identity will simplify the finite-difference bias terms below.

Let $x^*$ be the unique minimizer of $h$ and $X^*\coloneqq\mathbf1(x^*)^\top$. For a stepsize $\gamma>0$ and $a>0$, define the initial-error prefactor
\begin{equation}\label{eq:initial_error_prefactor}
\begin{aligned}
    \mathcal D(a)\coloneqq{}&
    \sum_{i=1}^n\mathbb E\|x_i^0-x^*\|^2
    +\frac{4\gamma^2\rho_J}{\sigma}
      \sum_{i=1}^n\sum_{j=1}^{m_i}
      \mathbb E\|s_{i,j}^0-\nabla f_{i,j}(x^*)\|^2\\
    &+\frac{\gamma^2}{a}
      \sum_{i=1}^n\|\nabla f_i(x^*)-\nabla f(x^*)\|^2.
\end{aligned}
\end{equation}

\begin{theorem}[Convergence rate]\label{thm:zivr_convergence_rate}
Consider Algorithm~\ref{alg:zivr_distributed} with a fixed finite-difference radius $\beta>0$.
Suppose Assumptions~\ref{assump:communication_network}, \ref{assump:basic_problem}, and~\ref{assump:sampling_operator} hold. Choose
\[
    0 < \gamma \le \frac{1}{m_{\max}L(\rho_0+4\rho_J)},
\]
and define
\[
    \tau\coloneqq \min\left\{\frac{\gamma\mu}{2},\,\frac{\sigma}{4}\right\},
    \qquad
    \delta\coloneqq \max\{1-\tau,\,1-\lambda_2(B)\}.
\]
Then $\delta<1$. Define the complete finite-difference error floor
\begin{equation}\label{eq:beta_error_floor}
    \mathcal B_\beta\coloneqq
    \frac{5\gamma^2dL^2(n+M\rho_J)\beta^2}
    {(1-\delta)^2}.
\end{equation}
For all $k\ge 1$,
\begin{equation}\label{eq:main_primal_rate}
    \mathbb{E}\|X^k-X^*\|_F^2
    \le
    \left(\frac{1+\delta}{2}\right)^{k-1}
    \mathcal D\bigl(\lambda_2(B)\bigr)
    +
    \mathcal B_\beta.
\end{equation}
\end{theorem}

We count each fresh component--direction pair as one oracle unit, requiring at most two function values, and measure communication by synchronous neighbor-exchange rounds.
\begin{corollary}[Oracle and communication costs without acceleration]
\label{cor:unaccelerated_complexity}
Under the conditions of Theorem~\ref{thm:zivr_convergence_rate}, suppose that, for every $i\in[n]$ and $k\ge0$,
\[
    m_i=m,\qquad S_i\coloneqq|\mathcal S_i^k|=R_i=R,
    \qquad R\in[md].
\]
Let $\kappa=L/\mu$ and $\rho=\|W-n^{-1}\mathbf1\mathbf1^\top\|_2<1$, and choose
\begin{equation}\label{eq:unaccelerated_batch_parameters}
    \gamma=\frac{R}{L(R+10d)},\qquad
    \tau=\min\left\{\frac{R}{2\kappa(R+10d)},
      \frac{R}{4md}\right\}.
\end{equation}
For a target $\epsilon>0$, choose
\begin{gather*}
    \beta^2\le
    \frac{\epsilon\min\{\tau,(1-\rho^2)/4\}^2}
    {10\gamma^2ndL^2(1+2d/R)},\\
    k=1+\left\lceil2\max\left\{\frac1\tau,\frac4{1-\rho^2}\right\}
    \log\max\left\{1,
    \frac{2\mathcal D((1-\rho^2)/4)}{\epsilon}\right\}\right\rceil.
\end{gather*}
Then $\mathbb E\|X^k-X^*\|_F^2\le\epsilon$. The oracle cost per agent and the number of communication rounds are, respectively,
\begin{gather*}
    \mathcal O\!\left(
      \left[\kappa(R+d)+md
        +\frac{R}{1-\rho^2}\right]
      \log\left(2+\frac{\mathcal D((1-\rho^2)/4)}{\epsilon}\right)\right),\\
    \mathcal O\!\left(
      \left[\kappa\left(1+\frac dR\right)
        +\frac{md}{R}+\frac1{1-\rho^2}\right]
      \log\left(2+\frac{\mathcal D((1-\rho^2)/4)}{\epsilon}\right)\right).
\end{gather*}
\end{corollary}

\begin{remark}[Unit batches without acceleration]
\label{rem:unaccelerated_unit_batch}
For $R=1$, corollary~\ref{cor:unaccelerated_complexity} gives both per-agent oracle cost and communication cost of
\[
    \widetilde{\mathcal{O}}\!\left(\left[d(m+\kappa)+\frac1{1-\rho^2}\right]
      \log\frac1{\epsilon}\right).
\]
The bound above is reported for Distributed \AlgAbbrv{} in Table~\ref{tab:related_work}.

Compared with the $\mathcal O(md^2\kappa^2(1-\rho^2)^{-2}\log(1/\epsilon))$ oracle and communication bounds of GT--SAGA--2pVRGE \cite{mu2026distributed},
the bound above has linear dependence on $d$ and additive dependence on $m$ and $\kappa$, with a separate additive network penalty.
The method also accommodates the shared proximal regularizer while retaining a constant maximum batch size. If the network is well-conditioned, i.e.
$(1-\rho^2)^{-1}=\mathcal O(d(m+\kappa))$, its per-agent oracle cost reduces to $\widetilde{\mathcal{O}}(d(m+\kappa)\log(1/\epsilon))$,
matching the centralized unaccelerated bounds of ZPDVR and ZIVR \cite{di2024double,zhang2026zivr} under the logarithmic conventions above.
\end{remark}

\begin{theorem}[Chebyshev-accelerated convergence rate]
\label{thm:zivr_chebyshev_complexity}
Suppose Assumptions~\ref{assump:communication_network}, \ref{assump:basic_problem}, and~\ref{assump:sampling_operator} hold.
Retain the general sampling constants in \eqref{eq:sampling_constants}, and choose $\gamma$ and $\tau$ as in Theorem~\ref{thm:zivr_convergence_rate}. Let
\[
    \rho\coloneqq
    \left\|W-\frac{1}{n}\mathbf1\mathbf1^\top\right\|_2<1,
\]
and suppose that $\rho$, or a common upper bound below one used in its place, is known to all agents.
Consider Algorithm~\ref{alg:zivr_distributed} with a fixed finite-difference radius $\beta>0$ and the communication step in Algorithm~\ref{alg:chebyshev_consensus}, using
\begin{equation}\label{eq:chebyshev_round_choice}
    K\coloneqq\left\lceil
      \max\left\{1,\sqrt{\frac{8\tau}{1-\rho^2}}\right\}
    \right\rceil
\end{equation}
rounds when $0<\rho<1$, and one ordinary round when $\rho=0$.

Then, for every $k\ge1$,
\begin{equation}
    \mathbb E\|X^k-X^*\|_F^2
    \le\left(1-\frac\tau2\right)^{k-1}
    \mathcal D(\tau)
    +\mathcal B_{{\rm acc},\beta},
    \label{eq:chebyshev_primal_rate}
\end{equation}
where
\begin{equation}\label{eq:accelerated_beta_error_floor}
    \mathcal B_{{\rm acc},\beta}
    \coloneqq\frac{5\gamma^2dL^2(n+M\rho_J)\beta^2}{\tau^2}.
\end{equation}

\end{theorem}

\begin{corollary}[Oracle and communication costs with Chebyshev acceleration]
\label{cor:accelerated_complexity}
Under the conditions of Theorem~\ref{thm:zivr_chebyshev_complexity}, suppose that, for every $i\in[n]$ and $k\ge0$,
\[
    m_i=m,\qquad S_i\coloneqq|\mathcal S_i^k|=R_i=R,
    \qquad R\in[md].
\]
Let $\kappa=L/\mu$ and choose
\begin{equation}\label{eq:accelerated_batch_parameters}
    \gamma=\frac{R}{L(R+10d)},\qquad
    \tau=\min\left\{\frac{R}{2\kappa(R+10d)},
      \frac{R}{4md}\right\}.
\end{equation}
For a target $\epsilon>0$, choose
\begin{gather*}
    \beta^2\le
    \frac{\epsilon\tau^2}{10\gamma^2ndL^2(1+2d/R)},\\
    k=1+\left\lceil\frac{2}{\tau}
    \log\max\left\{1,\frac{2\mathcal D(\tau)}{\epsilon}\right\}\right\rceil.
\end{gather*}
Then $\mathbb E\|X^k-X^*\|_F^2\le\epsilon$. With $K$ chosen by \eqref{eq:chebyshev_round_choice}, the oracle cost per agent is
\[
    \mathcal O\!\left(
      \left[\kappa(R+d)+md\right]
      \log\left(2+\frac{\mathcal D(\tau)}{\epsilon}\right)\right),
\]
and the number of communication rounds is
\[
    \mathcal O\!\left(\left[
      \kappa\left(1+\frac dR\right)+\frac{md}{R}
      +\sqrt{\frac{\kappa(1+d/R)+md/R}{1-\rho^2}}\right]
      \log\left(2+\frac{\mathcal D(\tau)}{\epsilon}\right)\right).
\]
\end{corollary}

\begin{remark}[Unit batches with Chebyshev acceleration]
\label{rem:accelerated_unit_batch}
For $R=1$, Corollary~\ref{cor:accelerated_complexity} gives per-agent oracle cost
\[
    \widetilde{\mathcal{O}}\!\left(d(m+\kappa)\log\frac1{\epsilon}\right)
\]
and communication cost
\[
    \widetilde{\mathcal{O}}\!\left(\left[d(m+\kappa)+
      \sqrt{\frac{d(m+\kappa)}{1-\rho^2}}\right]
      \log\frac1{\epsilon}\right).
\]
These are the rates reported for Distributed \AlgAbbrv{}--Cheb in Table~\ref{tab:related_work}.
Each iteration uses two component--direction pairs per agent, so both the average and maximum batch costs are $\mathcal O(1)$.
\end{remark}

\begin{remark}[Effect of communication acceleration]
\label{rem:chebyshev_general_complexity}
Corollary~\ref{cor:unaccelerated_complexity} shows that poor connectivity can increase both oracle and communication costs.
Chebyshev acceleration removes this spectral-gap factor from the oracle bound, at the cost of $K$ communication rounds per iteration. For fixed initialization, sampling parameters,
and stepsize, $\mathcal D(\tau)$ and the admissible smoothing radius in Corollary~\ref{cor:accelerated_complexity} have no explicit dependence on the communication matrix.

If the network is ill-conditioned, with $1-\rho^2\ll\tau$, acceleration improves the communication bound's network dependence from $(1-\rho^2)^{-1}$ to $(\tau(1-\rho^2))^{-1/2}$, apart from logarithmic factors.
For $1-\rho^2\ge\tau$, the ordinary algorithm already has an iteration bound of order $\tau^{-1}$, so acceleration need not improve the cost bounds.
\end{remark}

\section{Convergence Analysis}
\label{section:convergence_analysis}

We express Algorithm~\ref{alg:zivr_distributed} as an operator recursion and establish contraction for a reference update with exact directional observations.
A uniform bound on the finite-difference perturbation then gives the convergence rate.
For Chebyshev acceleration, we state the required spectral properties of the communication step and reuse the same estimates.
The supporting derivations and the oracle and communication cost proofs are collected in Appendix~\ref{section:technical_proofs}.
\subsection{Proof Setup and Auxiliary Lemmas}
\label{subsection:proof_setup_auxiliary_lemmas}

For $X=[x_1,\ldots,x_n]^\top\in\mathbb{R}^{n\times d}$, the global exact table is
\[
    J(X)\coloneqq [J_1(X)^\top,\ldots,J_n(X)^\top]^\top\in\mathbb{R}^{M\times d}.
\]
The corresponding global estimated table is
\[
    \tilde J^k\coloneqq
    [(\tilde J_1^k)^\top,\ldots,(\tilde J_n^k)^\top]^\top
    \in\mathbb{R}^{M\times d}.
\]
Define the matrix of local average gradients by
\[
    \bar G(X)\coloneqq
    \begin{bmatrix}
        \nabla f_1(x_1)^\top\\
        \vdots\\
        \nabla f_n(x_n)^\top
    \end{bmatrix}
    \in\mathbb R^{n\times d}.
\]

We use the sampling seed $\xi^k$ and filtration $\mathcal F_k$ defined in Section~\ref{section:algorithm_design}.
In fixed-input estimates, $\xi$ denotes the sampling seed of a single iteration, and expectations are taken with all inputs deterministic.
For arbitrary $X=[x_1,\ldots,x_n]^\top\in\mathbb R^{n\times d}$ and $\tilde J=[\tilde J_1^\top,\ldots,\tilde J_n^\top]^\top \in\mathbb R^{M\times d}$, write $s_{i,j}\coloneqq[\tilde J_i]_{j,:}^\top$.
Given a realization $\xi=\{\mathcal S_i,\mathcal R_i\}_{i=1}^n$ and a radius $\beta>0$, define the finite-difference estimator mapping rowwise by
\begin{equation}\label{eq:finite_difference_estimator_mapping}
    [\widehat G_{\xi}^{\beta}(X,\tilde J)]_{i,:}^\top
    \coloneqq\frac1{m_i}\sum_{j=1}^{m_i}s_{i,j}
    +\frac d{R_i}\sum_{(j,u)\in\mathcal R_i}
      \left(\widehat\nabla_u f_{i,j}(x_i;\beta)-uu^\top s_{i,j}\right),
    \qquad i\in[n].
\end{equation}
Define the accompanying table-refresh mapping by its agent blocks:
\begin{equation}\label{eq:finite_difference_table_mapping}
    [\mathcal J_{\xi}^{\beta}(X,\tilde J)_i]_{j,:}^\top
    \coloneqq s_{i,j}
    +\sum_{u:(j,u)\in\mathcal S_i}
      \left(\widehat\nabla_u f_{i,j}(x_i;\beta)-uu^\top s_{i,j}\right),
    \quad i\in[n],\ j\in[m_i],
\end{equation}
where an empty sum is zero. Along the iterates of Algorithm~\ref{alg:zivr_distributed}, these mappings satisfy
\[
    G^k\coloneqq[g_1^k,\ldots,g_n^k]^\top
    =\widehat G_{\xi^k}^{\beta^k}(X^k,\tilde J^k),
    \qquad
    \tilde J^{k+1}=\mathcal{J}_{\xi^k}^{\beta^k}(X^k,\tilde J^k).
\]

Define $\Psi(X)\coloneqq\sum_{i=1}^n\psi(x_i)$, so
\[
    \partial\Psi(X)
    =\{S\in\mathbb R^{n\times d}:[S]_{i,:}^\top\in\partial\psi(x_i)
      \text{ for every }i\in[n]\}.
\]
In particular, $v\in\partial\psi(x)$ implies $\mathbf{1}v^\top\in\partial\Psi(\mathbf{1}x^\top)$.

Set $c\coloneqq4\gamma^2\rho_J/\sigma$, and let $\mathcal{I}_J$ denote the identity mapping on the Jacobian-table space.
For any symmetric positive-semidefinite weight $D$, we write $\|V\|_D^2\coloneqq\operatorname{tr}(V^\top DV)$. Define the block seminorm by
\begin{equation}\label{eq:M_norm_definition}
    \|U\|_{\mathcal M}^2
    \coloneqq\|Z\|_A^2+\|Y\|_F^2+c\|\tilde J\|_F^2
\end{equation}
for $U=[Z^\top,Y^\top,\tilde J^\top]^\top$, where $Z,Y\in\mathbb R^{n\times d}$ and $\tilde J\in\mathbb R^{M\times d}$.

\begin{lemma}[Operator representation of Algorithm~\ref{alg:zivr_distributed}]
\label{lem:operator_splitting}

Suppose $\psi$ is proper, lower semicontinuous, and convex. Let $A,C\in\mathbb R^{n\times n}$ satisfy $A+C^2=I_n$,
and consider Algorithm~\ref{alg:zivr_distributed} with communication update $\zeta^{k+1}=A\chi^k$, where $\zeta^k$ and $\chi^k$ stack the corresponding local vectors rowwise.
For every realization of the sampling, initialize the global state by
\[
    Z^1 \coloneqq X^0-\gamma\widehat G_{\xi^0}^{\beta^0}(X^0,\tilde J^0),
    \qquad
    Y^1\coloneqq CZ^1,
\]
and set
\begin{equation}\label{eq:direct_global_recursion}
    U^1 \coloneqq [(Z^1)^\top, (Y^1)^\top, (\tilde J^1)^\top]^\top, \qquad 
    U^{k+1}\coloneqq\mathbf{T}_{\xi^k}^{\beta^k}(U^k),\qquad k\ge 1,
\end{equation}
where $\beta^k>0$ is the finite-difference radius at iteration $k$, and
\begin{equation}\label{eq:operator_decomposition_direct}
    \mathbf{T}_{\xi^k}^{\beta^k} \coloneqq
    \mathbf{T}_C
    \mathbf{T}_{f, \xi^k}^{\beta^k}
    \mathbf{T}_{\Psi,A},
\end{equation}
where
\begin{equation*}
    \mathbf{T}_C \coloneqq
    \begin{bmatrix}
    I_{n} & -C & 0 \\
    C & A & 0 \\
    0 & 0 & I_{M}
    \end{bmatrix}, \quad
    \mathbf{T}_{\Psi,A} \coloneqq
    \begin{bmatrix}
    \operatorname{prox}_{\gamma \Psi}\circ A & 0 & 0 \\
    0 & I_{n} & 0 \\
    0 & 0 & I_M
    \end{bmatrix}.
\end{equation*}
For any $X, Y \in \mathbb{R}^{n \times d}$ and $\tilde J \in \mathbb{R}^{M \times d}$, define
\begin{equation*}
  \mathbf{T}_{f, \xi^k}^{\beta^k}
    \begin{bmatrix}X\\Y\\\tilde J\end{bmatrix}
    \coloneqq \begin{bmatrix} 
      X - \gamma \widehat G_{\xi^{k}}^{\beta^k}(X, \tilde J)\\
      Y\\
      \mathcal{J}_{\xi^{k}}^{\beta^k}(X, \tilde J)
      \end{bmatrix}.
\end{equation*}
Here $\mathbf{T}_C$ is the communication operator, $\mathbf{T}_{f, \xi^k}^{\beta^k}$ is the stochastic gradient estimator/Jacobian-refresh operator, and $\mathbf{T}_{\Psi,A}$ is the combined communication/proximal operator.

Then, for every $k\ge1$, the state has the block form
\[
    U^k = [(Z^k)^\top, (Y^k)^\top, (\tilde J^k)^\top]^\top,
\]
where $\tilde J^k$ agrees with the algorithmic estimated table and $Z^k,Y^k\in\mathbb R^{n\times d}$. The primal iterate is recovered by
\[
    X^k = \operatorname{prox}_{\gamma\Psi}(AZ^k).
\]
\end{lemma}

The proof is given in Appendix~\ref{appendix:operator_fixed_point}.

\subsection{Contraction with Exact Directional Observations}
\label{subsection:exact_reference_contraction}
To isolate the finite-difference error, we introduce an auxiliary estimator that uses the exact directional projections $uu^\top\nabla f_{i,j}(x_i)$ with the same samples and stored table.
For a fixed input $X=[x_1,\ldots,x_n]^\top\in\mathbb R^{n\times d}$ and $\tilde J=[\tilde J_1^\top,\ldots,\tilde J_n^\top]^\top\in\mathbb R^{M\times d}$,
where $\tilde J_i\in\mathbb R^{m_i\times d}$, write $s_{i,j}=[\tilde J_i]_{j,:}^\top$ and define
\begin{align}
    [\widehat G_\xi^0(X,\tilde J)]_{i,:}^\top
    &\coloneqq\frac1{m_i}\sum_{j=1}^{m_i}s_{i,j}
      +\frac d{R_i}\sum_{(j,u)\in\mathcal R_i}
        uu^\top\bigl(\nabla f_{i,j}(x_i)-s_{i,j}\bigr),
      \label{eq:exact_reference_estimator}\\
    \mathcal J_\xi^0(X,\tilde J)
    &\coloneqq(\mathcal I_J-\Pi_\xi)\tilde J+\Pi_\xi J(X),
      \label{eq:exact_reference_table}
\end{align}
where $\Pi_\xi$ is the sampling projection in \eqref{eq:sampling_projection_definition} for the current seed $\xi$.
We suppress the iteration index in $\mathcal S_i$ and $\mathcal R_i=((j_{i,r},u_{i,r}))_{r=1}^{R_i}$; thus $r$ identifies an independent draw, including repeated component--direction pairs.
The superscript $0$ denotes exact directional observations, rather than evaluation of the finite-difference formula at $\beta=0$. The resulting reference operators are
\begin{equation}\label{eq:exact_reference_operator}
    \mathbf T_{f,\xi}^0
      \begin{bmatrix}X\\Y\\\tilde J\end{bmatrix}
    \coloneqq\begin{bmatrix}
      X-\gamma\widehat G_\xi^0(X,\tilde J)\\
      Y\\\mathcal J_\xi^0(X,\tilde J)
    \end{bmatrix},
    \qquad
    \mathbf T_\xi^0\coloneqq
      \mathbf T_C\mathbf T_{f,\xi}^0\mathbf T_{\Psi,A}.
\end{equation}
The operators $\mathbf T_{\xi}^{\beta}$ and $\mathbf T_\xi^0$ are always compared using the same seed and the same input. Only $\mathbf T_{\xi}^{\beta}$ is used by the algorithm.

\begin{lemma}[Pathwise fixed point of the reference operator]
\label{lem:reference_fixed_point}
Suppose the symmetric communication matrices satisfy
\begin{equation}\label{eq:analysis_communication_properties}
    0\preceq A\preceq I_n,\qquad B=I_n-A=C^2,\qquad
    C=B^{1/2},\qquad \operatorname{null}(B)=\operatorname{span}(\mathbf1).
\end{equation}
Let $x^*\in\argmin_x h(x)$, $X^*\coloneqq\mathbf1(x^*)^\top$, and $J^*\coloneqq J(X^*)$. There exist $Z^*$ and $Y^*\in\operatorname{range}(C)$ such that
\begin{equation}\label{eq:reference_fixed_point}
    U^*\coloneqq\begin{bmatrix}Z^*\\Y^*\\J^*\end{bmatrix},
    \qquad
    \operatorname{prox}_{\gamma\Psi}(AZ^*)=X^*,
    \qquad \mathbf T_\xi^0U^*=U^*
\end{equation}
for every realization of $\xi$.
\end{lemma}

The proof is given in Appendix~\ref{appendix:operator_fixed_point}.

In the following four lemmas, the two inputs are arbitrary fixed matrices
\[
    X^{(a)}=[x_1^{(a)},\ldots,x_n^{(a)}]^\top\in\mathbb R^{n\times d},
    \qquad
    \tilde J^{(a)}=[(\tilde J_1^{(a)})^\top,\ldots,
      (\tilde J_n^{(a)})^\top]^\top\in\mathbb R^{M\times d},
    \quad a=1,2,
\]
where $x_i^{(a)}\in\mathbb R^d$ and $\tilde J_i^{(a)}\in\mathbb R^{m_i\times d}$.
Both updates use the same random seed $\xi$ satisfying Assumption~\ref{assump:sampling_operator}; the expectations below hold these deterministic inputs fixed.

\begin{lemma}[Exact estimator difference]\label{lem:zivr_estimator_diff}
For $a=1,2$, let $g^{(a)}=\widehat G_\xi^0(X^{(a)},\tilde J^{(a)})$. Then
\begin{equation}\label{eq:exact_estimator_unbiasedness}
    \mathbb E g^{(a)}=\bar G(X^{(a)}),
\end{equation}
and
\begin{equation}\label{eq:zivr_estimator_diff_exact_table}
    \mathbb E\|g^{(1)}-g^{(2)}\|_F^2
    \le\rho_0\|J(X^{(1)})-J(X^{(2)})\|_F^2
      +\rho_J\|\tilde J^{(1)}-\tilde J^{(2)}\|_F^2.
\end{equation}
\end{lemma}

\begin{lemma}[Exact table-update difference]\label{lem:jacobian_update_diff}
For two fixed inputs, let $(\tilde J^{(a)})^+=\mathcal J_\xi^0(X^{(a)},\tilde J^{(a)})$, $a=1,2$, with the same seed $\xi$. Then
\begin{equation}\label{eq:jacobian_update_diff_exact_table}
\begin{aligned}
    \mathbb E\|(\tilde J^{(1)})^+-(\tilde J^{(2)})^+\|_F^2
    &=(1-\sigma)\|\tilde J^{(1)}-\tilde J^{(2)}\|_F^2\\
    &\quad+\sigma\|J(X^{(1)})-J(X^{(2)})\|_F^2.
\end{aligned}
\end{equation}
\end{lemma}

\begin{lemma}[Combined monotonicity and cocoercivity bound]
\label{lem:gradient_monotonicity_bound}
For two matrices $X^{(1)},X^{(2)}\in\mathbb{R}^{n\times d}$, define
\[
    \Delta X=X^{(1)}-X^{(2)},\qquad
    \Delta J=J(X^{(1)})-J(X^{(2)}),
    \qquad
    \Delta\bar G=\bar G(X^{(1)})-\bar G(X^{(2)}).
\]
Then
\begin{equation}\label{eq:gradient_monotonicity_bound}
    \left\langle \Delta X,\Delta\bar G\right\rangle_F
    \ge
    \frac{\mu}{2}\|\Delta X\|_F^2
    +\frac{1}{2m_{\max}L}\|\Delta J\|_F^2 .
\end{equation}
\end{lemma}

Proofs of Lemmas~\ref{lem:zivr_estimator_diff}--\ref{lem:gradient_monotonicity_bound} are given in Appendix~\ref{appendix:sampling_gradient_estimates}.

\begin{lemma}[Contraction of the exact gradient/table update]
\label{lem:global_contraction_Tf}
Choose $\gamma$ and $\tau$ as in Theorem~\ref{thm:zivr_convergence_rate}, retain $c=4\gamma^2\rho_J/\sigma$, and set $\theta=1-\tau$. For $a=1,2$, let
\[
    (X^{(a)})^+=X^{(a)}-\gamma\widehat G_\xi^0(X^{(a)},\tilde J^{(a)}),
    \qquad
    (\tilde J^{(a)})^+=\mathcal J_\xi^0(X^{(a)},\tilde J^{(a)}),
\]
using the same seed at both fixed inputs. Then
\begin{equation}\label{eq:global_contraction}
\begin{aligned}
    &\mathbb E\left[
      \|(X^{(1)})^+-(X^{(2)})^+\|_F^2
      +c\|(\tilde J^{(1)})^+-(\tilde J^{(2)})^+\|_F^2\right]\\
    &\quad\le\theta\left(
      \|X^{(1)}-X^{(2)}\|_F^2
      +c\|\tilde J^{(1)}-\tilde J^{(2)}\|_F^2\right).
\end{aligned}
\end{equation}
\end{lemma}

\begin{proof}
Write $\Delta X=X^{(1)}-X^{(2)}$, $\Delta J=J(X^{(1)})-J(X^{(2)})$, $\Delta\tilde J=\tilde J^{(1)}-\tilde J^{(2)}$, and $\Delta\bar G=\bar G(X^{(1)})-\bar G(X^{(2)})$.
Let $\Delta g=\widehat G_\xi^0(X^{(1)},\tilde J^{(1)}) -\widehat G_\xi^0(X^{(2)},\tilde J^{(2)})$. Because $\Delta X$ is fixed, unbiasedness in \eqref{eq:exact_estimator_unbiasedness} gives
\[
    \mathbb E\langle\Delta X,\Delta g\rangle_F
    =\langle\Delta X,\mathbb E\Delta g\rangle_F
    =\langle\Delta X,\Delta\bar G\rangle_F.
\]
Expanding the square therefore yields
\begin{align}
    \mathbb E\|\Delta X-\gamma\Delta g\|_F^2
    &=\|\Delta X\|_F^2-2\gamma\langle\Delta X,\Delta\bar G\rangle_F
      +\gamma^2\mathbb E\|\Delta g\|_F^2\notag\\
    &\le(1-\gamma\mu)\|\Delta X\|_F^2
      +\left(\gamma^2\rho_0-\frac\gamma{m_{\max}L}\right)\|\Delta J\|_F^2
      +\gamma^2\rho_J\|\Delta\tilde J\|_F^2.
      \label{eq:primal_contraction_rigorous}
\end{align}
The inequality uses the gradient bound \eqref{eq:gradient_monotonicity_bound} and the estimator bound \eqref{eq:zivr_estimator_diff_exact_table}.
Adding $c$ times the exact table identity \eqref{eq:jacobian_update_diff_exact_table} yields
\begin{align*}
    &\mathbb E\left[
      \|(X^{(1)})^+-(X^{(2)})^+\|_F^2
      +c\|(\tilde J^{(1)})^+-(\tilde J^{(2)})^+\|_F^2\right]\\
    &\quad\le(1-\gamma\mu)\|\Delta X\|_F^2
      +\bigl(\gamma^2\rho_J+c(1-\sigma)\bigr)\|\Delta\tilde J\|_F^2\\
    &\qquad\quad+
      \left(\gamma^2\rho_0+c\sigma-\frac\gamma{m_{\max}L}\right)\|\Delta J\|_F^2.
\end{align*}
The prescribed $c=4\gamma^2\rho_J/\sigma$ and stepsize imply
\begin{align*}
    \gamma^2\rho_J+c(1-\sigma)&=c\left(1-\frac{3\sigma}{4}\right),\\
    \gamma^2\rho_0+c\sigma-\frac\gamma{m_{\max}L}
    &=\gamma\left(\gamma(\rho_0+4\rho_J)-\frac1{m_{\max}L}\right)\le0.
\end{align*}
Thus the coefficient of $\|\Delta J\|_F^2$ is nonpositive. Moreover,
\[
    1-\gamma\mu\le1-2\tau\le1-\tau,\qquad
    1-\frac{3\sigma}{4}\le1-3\tau\le1-\tau,
\]
by $\tau\le\gamma\mu/2$ and $\tau\le\sigma/4$. Substitution gives
\begin{align*}
    &\mathbb E\left[
      \|(X^{(1)})^+-(X^{(2)})^+\|_F^2
      +c\|(\tilde J^{(1)})^+-(\tilde J^{(2)})^+\|_F^2\right]\\
    &\quad\le(1-\gamma\mu)\|\Delta X\|_F^2
      +c\left(1-\frac{3\sigma}{4}\right)\|\Delta\tilde J\|_F^2\\
    &\quad\le(1-\tau)\left(\|\Delta X\|_F^2+c\|\Delta\tilde J\|_F^2\right).
\end{align*}
\end{proof}

\begin{lemma}[Contraction of the reference operator]
\label{lem:reference_contraction}
Suppose the symmetric communication matrices satisfy \eqref{eq:analysis_communication_properties}.
Use the parameters of Lemma~\ref{lem:global_contraction_Tf} and set $\delta=\max\{\theta,1-\lambda_2(B)\}$. For any two fixed states $U^{(a)}=[(Z^{(a)})^\top,(Y^{(a)})^\top,(\tilde J^{(a)})^\top]^\top$,
$a=1,2$, where $Z^{(a)},Y^{(a)}\in\mathbb R^{n\times d}$ and $\tilde J^{(a)}\in\mathbb R^{M\times d}$, assume $Y^{(1)}-Y^{(2)}\in\operatorname{range}(C)$ columnwise. Then
\begin{equation}\label{eq:one_step_operator_rate}
    \mathbb E\|\mathbf T_\xi^0U^{(1)}-\mathbf T_\xi^0U^{(2)}\|_{\mathcal M}^2
    \le\delta\|U^{(1)}-U^{(2)}\|_{\mathcal M}^2.
\end{equation}
\end{lemma}

\begin{proof}
We use the operator-chain argument for the composite case in Xu et al.~\cite{xu2021distributed}.
Since $A=I_n-C^2$, the matrices $A$ and $C$ commute. For $Z,Y\in\mathbb R^{n\times d}$ and $H\in\mathbb R^{M\times d}$, we have
\[
    -2\langle Z,ACY\rangle_F+2\langle CZ,AY\rangle_F
    =2\langle Z,(CA-AC)Y\rangle_F=0.
\]
Therefore
\begin{align}
    \left\|\mathbf T_C\begin{bmatrix}Z\\Y\\H\end{bmatrix}\right\|_{\mathcal M}^2
    &=\|Z-CY\|_A^2+\|CZ+AY\|_F^2+c\|H\|_F^2\notag\\
    &=\|Z\|_{A+C^2}^2+\|Y\|_{CAC+A^2}^2+c\|H\|_F^2\notag\\
    &=\|Z\|_F^2+\|Y\|_A^2+c\|H\|_F^2.
      \label{eq:TC_energy_identity}
\end{align}
The last equality uses $A+C^2=I_n$ and $CAC+A^2=A(C^2+A)=A$.

Set $X^{(a)}=\operatorname{prox}_{\gamma\Psi}(AZ^{(a)})$ for $a=1,2$, and define
\begin{gather*}
    \Delta Z=Z^{(1)}-Z^{(2)},\quad
    \Delta Y=Y^{(1)}-Y^{(2)},\quad
    \Delta X=X^{(1)}-X^{(2)},\quad
    \Delta\tilde J=\tilde J^{(1)}-\tilde J^{(2)},\\
    \Delta X_+^0\coloneqq\Delta X-\gamma\bigl(
      \widehat G_\xi^0(X^{(1)},\tilde J^{(1)})
      -\widehat G_\xi^0(X^{(2)},\tilde J^{(2)})\bigr),\\
    \Delta\tilde J_+^0\coloneqq
      \mathcal J_\xi^0(X^{(1)},\tilde J^{(1)})
      -\mathcal J_\xi^0(X^{(2)},\tilde J^{(2)}).
\end{gather*}
Since the eigenvalues of $A$ lie in $[0,1]$, $A^2\preceq A$. Also, $\Delta Y\in\operatorname{range}(C)=\mathbf1^\perp$ columnwise gives
\begin{equation}\label{eq:deltaY_norm_bounds}
    \|\Delta Y\|_A^2
    =\|\Delta Y\|_F^2-\langle\Delta Y,B\Delta Y\rangle_F
    \le(1-\lambda_2(B))\|\Delta Y\|_F^2.
\end{equation}
The nonexpansiveness of the proximal map gives
\begin{equation}\label{eq:deltaX_norm_bounds}
    \|\Delta X\|_F^2\le\|A\Delta Z\|_F^2\le\|\Delta Z\|_A^2.
\end{equation}
Applying \eqref{eq:TC_energy_identity}, Lemma~\ref{lem:global_contraction_Tf}, and \eqref{eq:deltaY_norm_bounds} together with \eqref{eq:deltaX_norm_bounds}, in that order, yields
\begin{equation*}
\begin{aligned}
    \mathbb E\|\mathbf T_\xi^0U^{(1)}-\mathbf T_\xi^0U^{(2)}\|_{\mathcal M}^2
    &=\mathbb E\bigl[\|\Delta X_+^0\|_F^2
      +c\|\Delta\tilde J_+^0\|_F^2\bigr]+\|\Delta Y\|_A^2\\
    &\le\theta\bigl(\|\Delta X\|_F^2+c\|\Delta\tilde J\|_F^2\bigr)
      +\|\Delta Y\|_A^2\\
    &\le\theta\|\Delta Z\|_A^2+c\theta\|\Delta\tilde J\|_F^2
      +(1-\lambda_2(B))\|\Delta Y\|_F^2\\
    &\le\delta\|U^{(1)}-U^{(2)}\|_{\mathcal M}^2.
\end{aligned}
\end{equation*}
The last two inequalities use $\theta=1-\tau\ge3/4>0$ and $\delta\ge\theta$, $\delta\ge1-\lambda_2(B)$.
\end{proof}

\subsection{A Uniform Finite-Difference Perturbation Bound}
\label{subsection:uniform_perturbation}

We now compare the implemented update with the exact reference at the same input. The following observation bounds are uniform in the evaluation point, so no optimality condition is needed in this step.

\begin{lemma}[Forward-difference error]\label{lem:forward_difference_error}
Let $\varphi:\mathbb{R}^d\to\mathbb{R}$ be $L$-smooth and let $\|u\|=1$. Then, for every $x\in\mathbb{R}^d$ and $\beta>0$,
\begin{equation}\label{eq:forward_difference_error}
    \|\widehat\nabla_u \varphi(x;\beta)-uu^\top\nabla \varphi(x)\|
    \le \frac{L\beta}{2}.
\end{equation}

If, in addition, $u$ is uniformly distributed on $\mathbb S_d$, then
\begin{equation}\label{eq:forward_difference_mean_error}
    \left\|d\,\mathbb E_u\left[
    \widehat\nabla_u \varphi(x;\beta)-uu^\top\nabla \varphi(x)
    \right]\right\|^2
    \le \frac{dL^2\beta^2}{4}.
\end{equation}
\end{lemma}

\begin{lemma}[Uniform finite-difference perturbation]
\label{lem:uniform_perturbation}
Under \eqref{eq:analysis_communication_properties}, for every fixed input $U$ and every $\beta>0$,
\begin{equation}\label{eq:uniform_perturbation_bound}
    \mathbb E\|\mathbf T_{\xi}^{\beta}U-\mathbf T_\xi^0U\|_{\mathcal M}^2
    \le\mathcal E_\beta,
    \qquad
    \mathcal E_\beta\coloneqq\frac{L^2\beta^2}{4}
      \left[\gamma^2\sum_{i=1}^n\left(\frac{d^2}{R_i}+d\right)+c\nu\right].
\end{equation}
\end{lemma}

Proofs of Lemmas~\ref{lem:forward_difference_error} and~\ref{lem:uniform_perturbation}, together with the simplification in \eqref{eq:simplified_perturbation_bound}, are given in Appendix~\ref{appendix:finite_difference_estimates}.

\subsection{Proof of Theorem~\ref{thm:zivr_convergence_rate}}
\label{subsection:proof_theorem_one}

\begin{proof}[Proof of Theorem~\ref{thm:zivr_convergence_rate}]
The communication assumption implies \eqref{eq:analysis_communication_properties} and $\lambda_2(B)>0$. Since $0<\tau\le\sigma/4\le1/4$, $3/4\le\theta\le\delta<1$.
By Lemma~\ref{lem:operator_splitting}, the algorithm satisfies $U^{k+1}=\mathbf T_{\xi^k}^{\beta}U^k$ for $k\ge1$. Let $U^*$ be the fixed point constructed in Lemma~\ref{lem:reference_fixed_point}.
The algorithm's initialization $Y^1=CZ^1$ and the recursion $Y^{k+1}=Y^k+CZ^{k+1}$ in \eqref{eq:operator_splitting_y_recursion} imply $Y^k\in\operatorname{range}(C)$ for every $k\ge1$. Thus Lemma~\ref{lem:reference_contraction} applies to $(U^k,U^*)$.

The state $U^k$ is $\mathcal F_k$-measurable, and the fresh sampling law of Assumption~\ref{assump:sampling_operator} applies conditional on $\mathcal F_k$.
Using the pathwise fixed point in \eqref{eq:reference_fixed_point}, decompose the error as
\[
    U^{k+1}-U^*
    =\bigl(\mathbf T_{\xi^k}^0U^k-\mathbf T_{\xi^k}^0U^*\bigr)
      +\bigl(\mathbf T_{\xi^k}^{\beta}U^k-\mathbf T_{\xi^k}^0U^k\bigr).
\]
Set $\eta=(1-\delta)/(2\delta)>0$, so that $(1+\eta)\delta=(1+\delta)/2$ and $1+1/\eta=(1+\delta)/(1-\delta)$.
Applying Young's inequality to the error decomposition, followed by \eqref{eq:one_step_operator_rate} and \eqref{eq:uniform_perturbation_bound} conditional on $\mathcal F_k$, gives
\begin{equation}\label{eq:perturbed_conditional_recursion}
\begin{aligned}
    &\quad\mathbb E\!\left[\|U^{k+1}-U^*\|_{\mathcal M}^2
        \,\middle|\,\mathcal F_k\right]\\
    &\le(1+\eta)\delta\|U^k-U^*\|_{\mathcal M}^2
      +\left(1+\frac1\eta\right)\mathcal E_\beta\\
    &=\frac{1+\delta}{2}\|U^k-U^*\|_{\mathcal M}^2
      +\frac{1+\delta}{1-\delta}\mathcal E_\beta.
\end{aligned}
\end{equation}
The two terms in the error decomposition may be dependent: Young's inequality is applied pathwise before taking their expectations.

Set $r_k\coloneqq\mathbb E\|U^k-U^*\|_{\mathcal M}^2$. To bound the initial state, fix $0<a\le\lambda_2(B)$ and write
\[
    R_0\coloneqq\mathbb E\|X^0-X^*\|_F^2
      +c\mathbb E\|\tilde J^0-J^*\|_F^2.
\]
The moment condition \eqref{eq:initialization_moments} ensures $R_0<\infty$. Since $Y^*\in\operatorname{range}(C)$ and $A\preceq I_n$, \eqref{eq:original_stationary_y_construction} gives
\[
    \|Y^*\|_A^2\le\|Y^*\|_F^2
    \le\frac{\|CY^*\|_F^2}{a}
    =\frac{\gamma^2}{a}
      \sum_{i=1}^n\|\nabla f_i(x^*)-\nabla f(x^*)\|^2
    =\mathcal D(a)-R_0.
\]
The second inequality uses $C^2=B\succeq aI_n$ on $\mathbf1^\perp$. The initialization and fixed-point identities give
\[
    U^1=\mathbf T_C\begin{bmatrix}X^0-\gamma G^0\\0\\\tilde J^1\end{bmatrix},
    \qquad
    U^*=\mathbf T_C\begin{bmatrix}X^*-\gamma\bar G(X^*)\\Y^*\\J^*\end{bmatrix}.
\]
Apply Young's inequality with the same $\eta$ to the difference between the finite-difference and exact gradient/table updates at $(X^0,\tilde J^0)$.
Conditional on $(\zeta^0,\tilde J^0)$, Lemma~\ref{lem:global_contraction_Tf}, applied with the comparison input $(X^*,J^*)$, bounds the exact update,
and the gradient/table error estimates in the proof of Lemma~\ref{lem:uniform_perturbation} in Appendix~\ref{appendix:finite_difference_estimates} bound its perturbation by $\mathcal E_\beta$. Averaging and using the identity \eqref{eq:TC_energy_identity} yields
\begin{align*}
    r_1
    &=\mathbb E\!\left[
      \|X^0-X^*-\gamma(G^0-\bar G(X^*))\|_F^2
      +c\|\tilde J^1-J^*\|_F^2\right]+\|Y^*\|_A^2\\
    &\le(1+\eta)\theta R_0
      +\left(1+\frac1\eta\right)\mathcal E_\beta+\|Y^*\|_A^2\\
    &\le\mathcal D(a)
      +\frac{1+\delta}{1-\delta}\mathcal E_\beta.
\end{align*}
The last inequality uses $(1+\eta)\theta\le(1+\delta)/2<1$ and the bound on $\|Y^*\|_A^2$.

Taking total expectation in \eqref{eq:perturbed_conditional_recursion} gives $r_{k+1}\le(1+\delta)r_k/2+ (1+\delta)\mathcal E_\beta/(1-\delta)$ for $k\ge1$.
Iterating this recursion and substituting the bound on $r_1$ gives
\begin{align}
    r_k&\le\left(\frac{1+\delta}{2}\right)^{k-1}r_1+
      \frac{1+\delta}{1-\delta}\mathcal E_\beta
        \sum_{t=0}^{k-2}\left(\frac{1+\delta}{2}\right)^t\notag\\
    &\le\left(\frac{1+\delta}{2}\right)^{k-1}\mathcal D(a)+
      \frac{2(1+\delta)\mathcal E_\beta}{(1-\delta)^2}
        \left[1-\left(\frac{1+\delta}{2}\right)^k\right]\notag\\
    &\le\left(\frac{1+\delta}{2}\right)^{k-1}\mathcal D(a)+
      \frac{2(1+\delta)\mathcal E_\beta}{(1-\delta)^2}.
      \label{eq:perturbed_lyapunov_rate}
\end{align}
The second inequality evaluates the geometric sum after substituting the bound on $r_1$; the last uses $0<(1+\delta)/2<1$. The sum is empty at $k=1$, and $1-\delta>0$.
This argument uses only \eqref{eq:analysis_communication_properties} and $a\le\lambda_2(B)$, so it also applies to the accelerated matrices below. By \eqref{eq:simplified_perturbation_bound} and $1+\delta<2$,
\[
    \frac{2(1+\delta)\mathcal E_\beta}{(1-\delta)^2}
    \le\frac{4\mathcal E_\beta}{(1-\delta)^2}
    \le\frac{5\gamma^2dL^2(n+M\rho_J)\beta^2}{(1-\delta)^2}
    =\mathcal B_\beta.
\]
By the proximal identities for $X^k$ and $X^*$, nonexpansiveness and $A^2\preceq A$ give
\begin{equation}\label{eq:primal_lyapunov_comparison}
    \|X^k-X^*\|_F^2
    \le\|A(Z^k-Z^*)\|_F^2
    \le\|Z^k-Z^*\|_A^2
    \le\|U^k-U^*\|_{\mathcal M}^2.
\end{equation}
Taking expectations and choosing $a=\lambda_2(B)$ in the bound on $r_k$ gives
\[
    \mathbb E\|X^k-X^*\|_F^2
    \le r_k\le\left(\frac{1+\delta}{2}\right)^{k-1}
      \mathcal D\bigl(\lambda_2(B)\bigr)
      +\mathcal B_\beta.
\]
\end{proof}

\subsection{Proof of Theorem~\ref{thm:zivr_chebyshev_complexity}}
\label{subsection:proof_theorem_two}

We first state the spectral properties of the accelerated communication step, whose proof is given in Appendix~\ref{appendix:chebyshev_properties}. We then use these properties to transfer the operator estimates from Section~\ref{subsection:proof_theorem_one} to the accelerated state.

Let $T_r$ be the Chebyshev polynomials defined by
\[
    T_0(t)=1,\qquad T_1(t)=t,\qquad
    T_{r+1}(t)=2tT_r(t)-T_{r-1}(t),\qquad r\ge1.
\]
For $0<\rho<1$, set
\[
    P_K(t)\coloneqq\frac{T_K(t/\rho)}{T_K(1/\rho)}.
\]
When $\rho=0$, take $K=1$ and $P_1(t)=t$ to recover one ordinary communication round. In both cases, define
\begin{equation}\label{eq:chebyshev_matrices}
    A_K\coloneqq\frac{I_n+P_K(W)}{2},\qquad
    B_K\coloneqq I_n-A_K,\qquad
    C_K\coloneqq B_K^{1/2}.
\end{equation}
The following lemma verifies that $B_K$ is positive semidefinite, so its symmetric positive-semidefinite square root $C_K$ is well defined.

\begin{lemma}[Chebyshev communication and spectral gap]
\label{lem:chebyshev_spectral_gap}
Under the communication assumptions of Theorem~\ref{thm:zivr_chebyshev_complexity}, Algorithm~\ref{alg:chebyshev_consensus} computes $\zeta^{k+1}=A_K\chi^k$ in $K$ neighbor-exchange rounds.
The matrices in \eqref{eq:chebyshev_matrices} satisfy
\begin{equation}\label{eq:chebyshev_operator_properties}
    0\prec A_K\preceq I_n,\qquad
    0\preceq B_K\prec I_n,\qquad
    A_K+B_K=I_n,\qquad
    \operatorname{null}(B_K)=\operatorname{span}(\mathbf1).
\end{equation}
For $0<\rho<1$ and $K$ chosen by \eqref{eq:chebyshev_round_choice},
\begin{equation}\label{eq:chebyshev_gap_bound}
    \lambda_2(B_K)
    \ge\frac{1-T_K(1/\rho)^{-1}}{2}
    \ge\frac{K^2(1-\rho^2)}{4+2K^2(1-\rho^2)}\ge\tau.
\end{equation}
For $\rho=0$ and $K=1$, the gap equals $1/2\ge\tau$.
\end{lemma}

The proof is given in Appendix~\ref{appendix:chebyshev_properties}.

\begin{proof}[Proof of Theorem~\ref{thm:zivr_chebyshev_complexity}]
In this proof and the cost calculation in Appendix~\ref{appendix:complexity_proofs}, use $(A_K,B_K,C_K)$ in place of $(A,B,C)$ in the operators of Lemma~\ref{lem:operator_splitting}. For the analysis, the corresponding state and block seminorm are
\[
    U^k\coloneqq[(Z^k)^\top,(Y^k)^\top,(\tilde J^k)^\top]^\top,
    \qquad
    \|U\|_{\mathcal M}^2
    \coloneqq\|Z\|_{A_K}^2+\|Y\|_F^2+c\|\tilde J\|_F^2,
\]
where $U=[Z^\top,Y^\top,\tilde J^\top]^\top$. Lemma~\ref{lem:chebyshev_spectral_gap} gives $\zeta^{k+1}=A_K\chi^k$, including at $k=0$,
and verifies \eqref{eq:analysis_communication_properties} for these matrices. Lemma~\ref{lem:operator_splitting}, applied with $A_K$ and $C_K$, gives
\begin{equation}\label{eq:accelerated_state_recursion}
    U^{k+1}=\mathbf T_{\xi^k}^{\beta}U^k,\qquad
    X^k=\operatorname{prox}_{\gamma\Psi}(A_KZ^k),\qquad
    Y^{k+1}=Y^k+C_KZ^{k+1},\qquad k\ge1,
\end{equation}
where the operators use $A_K$ and $C_K$, and $Z^1=X^0-\gamma G^0$, $Y^1=C_KZ^1$. In particular, $Y^k\in\operatorname{range}(C_K)$ for every $k\ge1$.

Lemma~\ref{lem:reference_fixed_point}, applied with $(A_K,B_K,C_K)$, provides $U^*=[(Z^*)^\top,(Y^*)^\top,(J^*)^\top]^\top$ with $Y^*\in\operatorname{range}(C_K)$ and
\[
    \operatorname{prox}_{\gamma\Psi}(A_KZ^*)=X^*,
    \qquad \mathbf T_\xi^0U^*=U^*.
\]

By the spectral gap in Lemma~\ref{lem:chebyshev_spectral_gap},
\[
    \delta_K\coloneqq\max\{1-\tau,1-\lambda_2(B_K)\}=1-\tau.
\]
Lemmas~\ref{lem:reference_contraction} and~\ref{lem:uniform_perturbation} therefore give the same reference contraction and uniform perturbation bound with $\delta_K$ in place of $\delta$ and the same $\mathcal E_\beta$.
These lemmas use only \eqref{eq:analysis_communication_properties}; entrywise nonnegativity of $P_K(W)$ is unnecessary.
Apply \eqref{eq:perturbed_lyapunov_rate} with $a=\tau\le\lambda_2(B_K)$ and $\delta=\delta_K$. Since $(1+\delta_K)/2=1-\tau/2$ and $1-\delta_K=\tau$,
\begin{align*}
    \mathbb E\|U^k-U^*\|_{\mathcal M}^2
    &\le\left(1-\frac\tau2\right)^{k-1}
      \mathcal D(\tau)
      +\frac{4(1-\tau/2)\mathcal E_\beta}{\tau^2}\\
    &\le\left(1-\frac\tau2\right)^{k-1}
      \mathcal D(\tau)
      +\mathcal B_{{\rm acc},\beta}.
\end{align*}
For the last inequality, \eqref{eq:simplified_perturbation_bound} gives
\[
    \frac{4(1-\tau/2)\mathcal E_\beta}{\tau^2}
    \le\frac{4\mathcal E_\beta}{\tau^2}
    \le\frac{5\gamma^2dL^2(n+M\rho_J)\beta^2}{\tau^2}
    =\mathcal B_{{\rm acc},\beta}.
\]
The first inequality uses $1-\tau/2\le1$; the second uses \eqref{eq:simplified_perturbation_bound}. The primal comparison \eqref{eq:primal_lyapunov_comparison}, with $A_K^2\preceq A_K$, gives
\[
    \mathbb E\|X^k-X^*\|_F^2
    \le\mathbb E\|U^k-U^*\|_{\mathcal M}^2
    \le\left(1-\frac\tau2\right)^{k-1}
      \mathcal D(\tau)+\mathcal B_{{\rm acc},\beta}.
\]
Finally, the communication subroutine uses $K$ neighbor-exchange rounds and leaves the function queries in the gradient and table updates unchanged.
\end{proof}

The proofs of Corollaries~\ref{cor:unaccelerated_complexity} and~\ref{cor:accelerated_complexity} are given in Appendix~\ref{appendix:complexity_proofs}.

\section{Numerical Experiments}
\label{section:numerical_experiments}

We evaluate Distributed ZIVR and its Chebyshev-accelerated variant on two real-data problems with elastic-net regularization: logistic regression on handwritten digits and robust regression on airfoil measurements.
We compare them with proximal adaptations of DGD--2p and GT--2d~\cite{tang2021distributed}, VR--GE~\cite{mu2024variance}, and GT--SAGA--2pVRGE~\cite{mu2026distributed}.
For DGD--2p, GT--2d, and VR--GE, which do not exploit the local finite-sum structure, each agent treats its local average $f_i=m_i^{-1}\sum_{j=1}^{m_i}f_{i,j}$ as a single smooth function.
Each evaluation of $f_i$ is implemented by evaluating and averaging all $m_i$ component functions and is charged $m_i$ component-oracle calls. Distributed ZIVR and GT--SAGA--2pVRGE instead access the component functions individually.

All methods start from zero and communicate over an undirected path graph with Metropolis weights. They use the same fixed smoothing radius $\beta=10^{-9}$ and ridge coefficient $\lambda_2=0.1$.
Each objective additionally contains $\psi(x)=\lambda_1\|x\|_1$ with $\lambda_1=0.005$ for both problems, handled by the exact soft-thresholding proximal map.
The function-value oracles evaluate the smooth component losses. Stepsizes are tuned separately for each method using the final portion of its tuning trajectory. One Chebyshev configuration is used for each problem, with $36$ mixing rounds for Digits and $8$ for Airfoil.

We report the mean squared solution error and consensus error,
\[
    E^k=\frac{1}{n}\sum_{i=1}^{n}\|x_i^k-x_{\mathrm{ref}}\|^2,
    \qquad
    C^k=\frac{1}{n}\sum_{i=1}^{n}\|x_i^k-\bar{x}^k\|^2,
\]
where $\bar{x}^k=\frac{1}{n}\sum_i x_i^k$ and $x_{\mathrm{ref}}$ is an accurately computed numerical minimizer of the regularized objective.
Oracle calls count component-function evaluations, reported as the largest cumulative count among agents, including initialization. Communication counts include every mixing round.

\subsection{Logistic Regression on Handwritten Digits}
\label{subsec:numerical_digits}

We classify odd versus even digits using the Digits dataset~\cite{alpaydin1998optdigits}. The training set contains $1536$ images, distributed over $64$ agents with $24$ distinct observations per agent.
The $64$ pixel intensities are scaled to $[0,1]$, and an intercept is appended, giving $d=65$.
Local datasets are class-skewed, with approximately three quarters of each agent's observations coming from one parity class.
With augmented feature vector $a_{i,j}$ and label $y_{i,j}\in\{-1,1\}$, we solve
\begin{equation}\label{eq:numerical_digits_objective}
    \min_{x\in\mathbb R^d}\quad h(x)=
    \frac{1}{n}\sum_{i=1}^{n}\frac{1}{m_i}\sum_{j=1}^{m_i}
    \log\!\left(1+\exp(-y_{i,j}a_{i,j}^{\top}x)\right)
    +\frac{\lambda_2}{2}\|x\|^2+\lambda_1\|x\|_1.
\end{equation}
This is an instance of~\eqref{eq:formulation}, with $f_{i,j}$ equal to the logistic loss of observation $(a_{i,j},y_{i,j})$ plus $\lambda_2\|x\|^2/2$, and $\psi(x)=\lambda_1\|x\|_1$.
The component functions are convex and smooth, the ridge term makes each local average $\lambda_2$-strongly convex, and the $\ell_1$ term promotes sparsity in the fitted coefficients.

\begin{figure}[!tb]
    \centering
    \begin{minipage}[t]{0.325\linewidth}
        \centering
        \includegraphics[width=\linewidth]{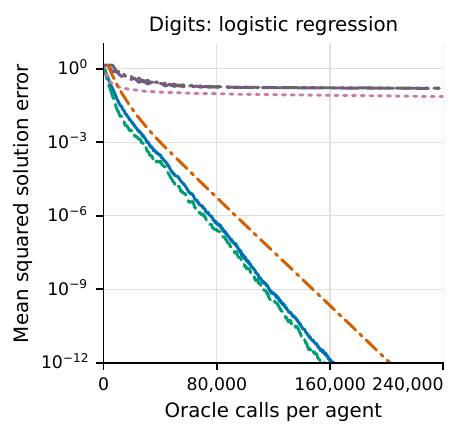}
        \caption{Digits: solution error.}
        \label{fig:digits_oracle_calls}
    \end{minipage}\hfill
    \begin{minipage}[t]{0.325\linewidth}
        \centering
        \includegraphics[width=\linewidth]{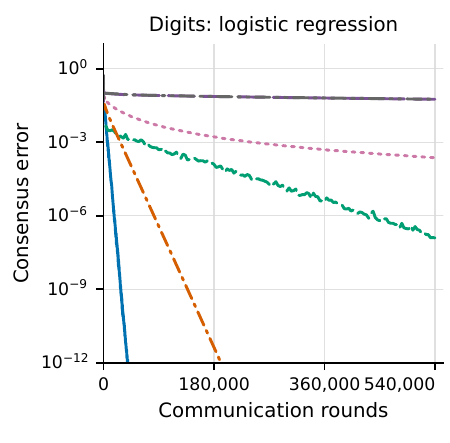}
        \caption{Digits: consensus error vs. communications.}
        \label{fig:digits_consensus_communications}
    \end{minipage}\hfill
    \begin{minipage}[t]{0.325\linewidth}
        \centering
        \includegraphics[width=\linewidth]{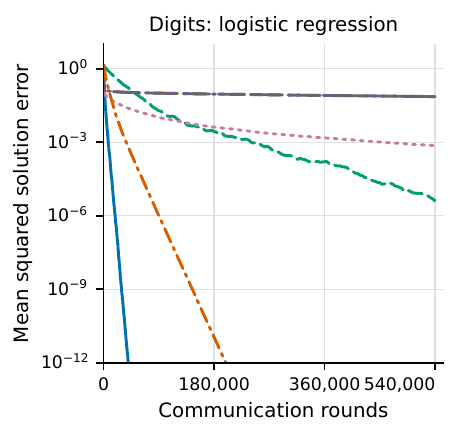}
        \caption{Digits: communications.}
        \label{fig:digits_communications}
    \end{minipage}

    \par\medskip
    \begin{minipage}[t]{0.325\linewidth}
        \centering
        \includegraphics[width=\linewidth]{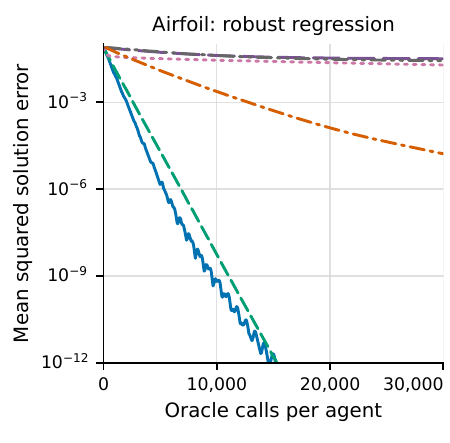}
        \caption{Airfoil: solution error.}
        \label{fig:airfoil_oracle_calls}
    \end{minipage}\hfill
    \begin{minipage}[t]{0.325\linewidth}
        \centering
        \includegraphics[width=\linewidth]{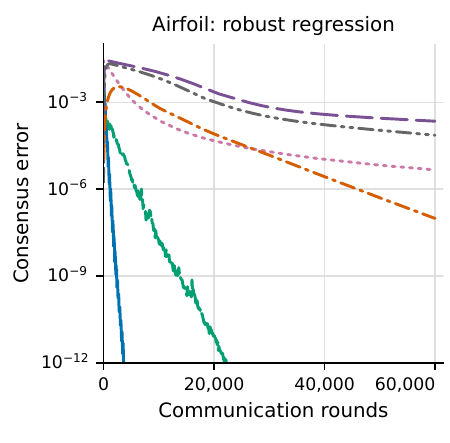}
        \caption{Airfoil: consensus error vs. communications.}
        \label{fig:airfoil_consensus_communications}
    \end{minipage}\hfill
    \begin{minipage}[t]{0.325\linewidth}
        \centering
        \includegraphics[width=\linewidth]{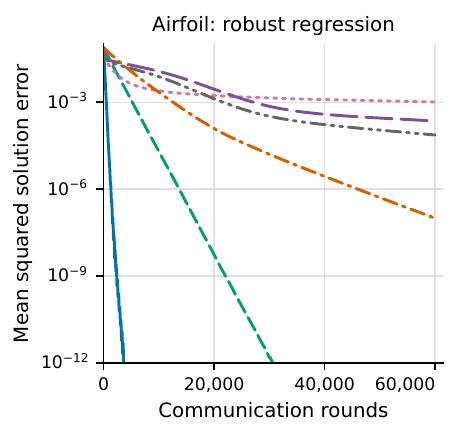}
        \caption{Airfoil: communications.}
        \label{fig:airfoil_communications}
    \end{minipage}
    \par\smallskip
    \includegraphics[width=\linewidth]{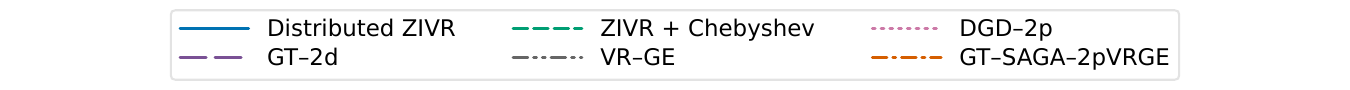}
\end{figure}
Figures~\ref{fig:digits_oracle_calls}--\ref{fig:digits_communications} 
illustrate the benefit of combining incremental gradient estimation with a correction for disagreement among agents. After the initial transient, Distributed ZIVR exhibits a linear convergence behavior which is consistent with the convergence result for strongly convex composite objectives.
DGD–2p uses inexpensive directional estimates, but their variance demands the use of a diminishing stepsize regime which slows down convergence. GT–2d obtains more accurate estimates at a substantially greater query cost, while VR–GE also incurs the cost of evaluating complete local averages.
 GT–SAGA–2pVRGE, which also combines incremental variance reduction with gradient tracking, outperforms the remaining baselines but converges more slowly than both Distributed ZIVR variants in this experiment.
 Chebyshev mixing slightly improves progress per oracle call, but its additional exchanges outweigh this gain when communication is counted. Thus, for this configuration, ordinary mixing already provides sufficient information exchange to support effective local optimization.

\subsection{Robust Regression on Airfoil Measurements}
\label{subsec:numerical_airfoil}

The Airfoil Self-Noise dataset~\cite{brooks1989airfoil} relates five physical measurements to sound pressure.
We use $1280$ training observations, partitioned over $80$ agents with $16$ distinct observations each. Features and responses are standardized using the training data, and an intercept gives $d=6$. Using the smooth, nonquadratic pseudo-Huber loss with $\delta=0.1$, we solve
\begin{equation}\label{eq:numerical_airfoil_objective}
    \min_{x\in\mathbb R^d}\quad h(x) = 
    \frac{1}{n}\sum_{i=1}^{n}\frac{1}{m_i}\sum_{j=1}^{m_i}
    \delta\left(\sqrt{\delta^2+(a_{i,j}^{\top}x-y_{i,j})^2}-\delta\right)
    +\frac{\lambda_2}{2}\|x\|^2+\lambda_1\|x\|_1.
\end{equation}
Here $f_{i,j}$ in~\eqref{eq:formulation} is the pseudo-Huber loss of one observation plus the ridge term, and $\psi(x)=\lambda_1\|x\|_1$, with the same elastic-net coefficients as above.
The pseudo-Huber loss is convex and smooth, so the ridge term again ensures $\lambda_2$-strong convexity of each local average, while the nonsmooth penalty is handled proximally.

Figures~\ref{fig:airfoil_oracle_calls}--\ref{fig:airfoil_communications} show that both ZIVR variants rapidly reduce solution error with respect to oracle calls.
Distributed ZIVR reduces both solution and consensus errors with fewer communication rounds than its Chebyshev variant.
GT--SAGA--2pVRGE improves more slowly; the other baselines retain larger errors over the comparison budgets.

The limited oracle-efficiency gains from Chebyshev acceleration are consistent with the rate bounds.
Both ZIVR variants use $R_i=|\mathcal S_i^k|=1$, namely one component--direction pair for the estimator and one for the table refresh, and communicate after every such iteration.
With only two fresh pairs per agent between communication steps, local estimation progresses slowly relative to mixing: the path graphs are not poorly connected enough for the network term to dominate $\min\{\tau,\lambda_2(B)\}$ in Theorem~\ref{thm:zivr_convergence_rate}.
In this regime, the choice~\eqref{eq:chebyshev_round_choice} can already give $K=1$ when $8\tau\le 1-\rho^2$, so additional mixing does not improve the oracle-complexity order.
The larger fixed degrees tested here may improve transient behavior, as on Digits, but incur extra communication without a substantial reduction in oracle calls.

\section{Conclusion}
\label{section:conclusion}

We proposed Distributed \AlgAbbrv{}, combining incremental zeroth-order variance reduction with proximal gradient tracking for decentralized composite finite-sum optimization at constant per-iteration oracle cost, without full refreshes.
For strongly convex local averages, we established linear convergence up to a controllable finite-difference error.
With Chebyshev mixing, the method matches centralized oracle complexity with inverse square-root spectral-gap dependence in communication.
Numerical experiments support its effectiveness.

Future directions include extending the method to convex problems without strong convexity and developing Nesterov-type acceleration, potentially via Katyusha-style momentum~\cite{allen2018katyusha}, to improve condition-number dependence while retaining constant per-iteration oracle cost.

\paragraph{Statement on AI use.}
The authors used ChatGPT 5.6 Sol to improve the exposition of the proofs. Notably, The authors adopted the AI system’s suggestion to simplify the convergence analysis by first analyzing an exact reference operator and treating finite-difference errors as a perturbation. In addition, the AI system was used for polishing the manuscript's language.
All AI-assisted revisions were reviewed and verified by the authors, who take full responsibility for the manuscript's content and accuracy.

\bibliographystyle{IEEEtran}
\bibliography{main_bib}

\appendix
\section{Technical Proofs}
\label{section:technical_proofs}

This appendix contains the supporting proofs for Section~\ref{section:convergence_analysis} and the oracle and communication cost derivations for Corollaries~\ref{cor:unaccelerated_complexity} and~\ref{cor:accelerated_complexity}.
The notation and assumptions are those used in Section~\ref{section:convergence_analysis}.

\subsection{Operator Representation and Fixed-Point Construction}
\label{appendix:operator_fixed_point}

\begin{proof}[Proof of Lemma~\ref{lem:operator_splitting}]
Let $G^k$ be the stacked algorithmic estimator. We prove by induction that the last block of $U^k$ is $\tilde J^k$ and that
\begin{equation}\label{eq:operator_splitting_auxiliary_identities}
    \zeta^k=AZ^k,
    \qquad
    X^k=\operatorname{prox}_{\gamma\Psi}(AZ^k),
    \qquad
    CY^k=X^{k-1}-\gamma G^{k-1}-AZ^k.
\end{equation}

At $k=1$, the table correspondence holds by the definition of $U^1$, and
\begin{align*}
    \zeta^1&=A(X^0-\gamma G^0)=AZ^1,\\
    X^1&=\operatorname{prox}_{\gamma\Psi}(\zeta^1)
       =\operatorname{prox}_{\gamma\Psi}(AZ^1),\\
    CY^1&=C^2Z^1=(I_n-A)Z^1=X^0-\gamma G^0-AZ^1.
\end{align*}
The first two lines use the algorithm's initialization and rowwise proximal update; the last uses $Y^1=CZ^1$ and $C^2=I_n-A$.

Suppose the correspondence and \eqref{eq:operator_splitting_auxiliary_identities} hold at $k\ge1$. Substitution into the three operators gives
\begin{equation}\label{eq:operator_splitting_block_recursion}
    U^{k+1}=\begin{bmatrix}
        X^k-\gamma G^k-CY^k\\
        C(X^k-\gamma G^k)+AY^k\\
        \mathcal J_{\xi^k}^{\beta^k}(X^k,\tilde J^k)
    \end{bmatrix}.
\end{equation}
The last block is $\tilde J^{k+1}$ by the algorithm's table update. The first two blocks and $C^2+A=I_n$ give
\begin{align}
    Y^{k+1}
    &=C(Z^{k+1}+CY^k)+AY^k\notag\\
    &=Y^k+CZ^{k+1}.
    \label{eq:operator_splitting_y_recursion}
\end{align}
Using the last identity in \eqref{eq:operator_splitting_auxiliary_identities} in the first block gives
\[
    Z^{k+1}
    =AZ^k+X^k-X^{k-1}-\gamma(G^k-G^{k-1})=\chi^k,
\]
where the last equality uses $AZ^k=\zeta^k$ and the algorithm's history correction. Its communication and proximal steps therefore yield
\begin{align*}
    \zeta^{k+1}&=A\chi^k=AZ^{k+1},\\
    X^{k+1}&=\operatorname{prox}_{\gamma\Psi}(\zeta^{k+1})
           =\operatorname{prox}_{\gamma\Psi}(AZ^{k+1}).
\end{align*}
Finally, \eqref{eq:operator_splitting_y_recursion}, $C^2=I_n-A$, and the first block of \eqref{eq:operator_splitting_block_recursion} give
\begin{align*}
    CY^{k+1}
    &=CY^k+(I_n-A)Z^{k+1}\\
    &=X^k-\gamma G^k-Z^{k+1}+(I_n-A)Z^{k+1}\\
    &=X^k-\gamma G^k-AZ^{k+1}.
\end{align*}
This completes the induction.
\end{proof}

\begin{proof}[Proof of Lemma~\ref{lem:reference_fixed_point}]
Since $f$ is finite and differentiable on $\mathbb R^d$, the convex subdifferential sum rule and optimality of $x^*$ give
\[
    0\in\partial h(x^*)=\nabla f(x^*)+\partial\psi(x^*),
    \qquad -\frac1n\sum_{i=1}^n\nabla f_i(x^*)\in\partial\psi(x^*).
\]
Hence
\begin{equation}\label{eq:original_lifted_subgradient_choice}
    S^*\coloneqq-\mathbf1\left(\frac1n\mathbf1^\top\bar G(X^*)\right)
      \in\partial\Psi(X^*),
    \qquad \mathbf1^\top\bigl(\bar G(X^*)+S^*\bigr)=0.
\end{equation}
Set $Z^*=X^*+\gamma S^*$. Since $\operatorname{range}(C)=\mathbf1^\perp$ columnwise, there is a unique $Y^*\in\operatorname{range}(C)$ satisfying
\begin{equation}\label{eq:original_stationary_y_construction}
    CY^*=-\gamma\bigl(\bar G(X^*)+S^*\bigr).
\end{equation}
Both $X^*$ and $S^*$ are consensual, so $AZ^*=Z^*$ and $CZ^*=0$. The proximal characterization therefore gives
\[
    AZ^*-X^*=\gamma S^*\in\gamma\partial\Psi(X^*),
    \qquad \operatorname{prox}_{\gamma\Psi}(AZ^*)=X^*.
\]
At $(X^*,J^*)$, $[J_i(X^*)]_{j,:}^\top=\nabla f_{i,j}(x^*)$, so for every $(j,u)\in\mathcal R_i$,
\[
    uu^\top\bigl(\nabla f_{i,j}(x^*)-[J_i(X^*)]_{j,:}^\top\bigr)=0.
\]
Consequently, \eqref{eq:exact_reference_estimator}--\eqref{eq:exact_reference_table} give, for every $\xi$,
\begin{align*}
    [\widehat G_\xi^0(X^*,J^*)]_{i,:}^\top
    &=\frac1{m_i}\sum_{j=1}^{m_i}\nabla f_{i,j}(x^*)=\nabla f_i(x^*),\\
    \mathcal J_\xi^0(X^*,J^*)
    &=(\mathcal I_J-\Pi_\xi)J^*+\Pi_\xi J^*=J^*.
\end{align*}
Using \eqref{eq:original_stationary_y_construction} in the first block and $C^2+A=I_n$ in the second gives
\[
    \mathbf T_\xi^0U^*
    =\begin{bmatrix}
      X^*-\gamma\bar G(X^*)-CY^*\\
      C\bigl(X^*-\gamma\bar G(X^*)\bigr)+AY^*\\J^*
    \end{bmatrix}
    =\begin{bmatrix}
      Z^*\\CZ^*+(C^2+A)Y^*\\J^*
    \end{bmatrix}
    =U^*.
\]
\end{proof}

\subsection{Sampling and Gradient Estimates}
\label{appendix:sampling_gradient_estimates}

\begin{proof}[Proof of Lemma~\ref{lem:zivr_estimator_diff}]
Uniform component sampling and $d\mathbb E_u[uu^\top]=I_d$ give, for every fixed $(X,\tilde J)$,
\[
    [\mathbb E\widehat G_\xi^0(X,\tilde J)]_{i,:}^\top
    =\frac1{m_i}\sum_{j=1}^{m_i}s_{i,j}
      +\frac d{m_i}\sum_{j=1}^{m_i}\mathbb E_u[uu^\top]
        \bigl(\nabla f_{i,j}(x_i)-s_{i,j}\bigr)
    =\nabla f_i(x_i).
\]
This proves \eqref{eq:exact_estimator_unbiasedness}. For later use, let $R\ge1$ be an integer.
If $V_1,\ldots,V_R\in\mathbb R^d$ are square-integrable i.i.d. vectors and $a\in\mathbb R^d$ is deterministic, then
\begin{equation}\label{eq:iid_batch_second_moment}
\begin{aligned}
    \mathbb E\left\|a+\frac1R\sum_{r=1}^R(V_r-\mathbb E V_1)\right\|^2
    &=\|a\|^2+\frac1{R^2}\sum_{r=1}^R
      \mathbb E\|V_r-\mathbb E V_1\|^2\\
    &=\|a\|^2+\frac1R\left(\mathbb E\|V_1\|^2-\|\mathbb E V_1\|^2\right).
\end{aligned}
\end{equation}
The first equality uses zero means and independence to eliminate the cross terms; the second uses identical distributions and expands the centered square.

For the difference bound, write $g_i^{(a)}=[g^{(a)}]_{i,:}^\top$ and set
\[
    \Delta J_i=J_i(X^{(1)})-J_i(X^{(2)}),\qquad
    \Delta\tilde J_i=\tilde J_i^{(1)}-\tilde J_i^{(2)},\qquad
    v_{i,j}=[\Delta J_i-\Delta\tilde J_i]_{j,:}^\top.
\]
Define the component mean by
\[
    \bar v_i\coloneqq\frac1{m_i}\sum_{j=1}^{m_i}v_{i,j}.
\]
For $\Delta\bar g_i\coloneqq\nabla f_i(x_i^{(1)})-\nabla f_i(x_i^{(2)})$, \eqref{eq:exact_reference_estimator} gives
\[
    g_i^{(1)}-g_i^{(2)}
    =\Delta\bar g_i+\frac1{R_i}\sum_{r=1}^{R_i}
      \left(d u_{i,r}u_{i,r}^\top v_{i,j_{i,r}}-\bar v_i\right),
\]
and uniform sampling gives
\[
    \mathbb E_{j,u}[d uu^\top v_{i,j}]
    =\frac d{m_i}\sum_{j=1}^{m_i}\mathbb E_u[uu^\top]v_{i,j}=\bar v_i.
\]
Applying \eqref{eq:iid_batch_second_moment} with $V_r=d u_{i,r}u_{i,r}^\top v_{i,j_{i,r}}$ and $a=\Delta\bar g_i$ yields
\begin{align*}
    \mathbb E\|g_i^{(1)}-g_i^{(2)}\|^2
    &=\|\Delta\bar g_i\|^2
      +\frac1{R_i}\left(\frac{d^2}{m_i}\sum_{j=1}^{m_i}
        \mathbb E_u(u^\top v_{i,j})^2-\|\bar v_i\|^2\right)\\
    &=\|\Delta\bar g_i\|^2
      +\frac1{R_i}\left(\frac d{m_i}
        \|\Delta J_i-\Delta\tilde J_i\|_F^2-\|\bar v_i\|^2\right)\\
    &\le\|\Delta\bar g_i\|^2
      +\frac{2d}{m_iR_i}\bigl(\|\Delta J_i\|_F^2+\|\Delta\tilde J_i\|_F^2\bigr).
\end{align*}
The first equality also uses $\|u\|=1$; the second uses $\mathbb E_u[uu^\top]=I_d/d$.
The inequality drops $-\|\bar v_i\|^2$ and applies $\|\Delta J_i-\Delta\tilde J_i\|_F^2 \le2\|\Delta J_i\|_F^2+2\|\Delta\tilde J_i\|_F^2$. Jensen's inequality gives
\[
    \|\Delta\bar g_i\|^2
    =\left\|\frac1{m_i}\sum_{j=1}^{m_i}[\Delta J_i]_{j,:}^\top\right\|^2
    \le\frac1{m_i}\sum_{j=1}^{m_i}\|[\Delta J_i]_{j,:}^\top\|^2
    =\frac1{m_i}\|\Delta J_i\|_F^2.
\]
Substituting the last two bounds and summing over agents yields
\begin{align*}
    \mathbb E\|g^{(1)}-g^{(2)}\|_F^2
    &\le\sum_{i=1}^n\left[
      \left(\frac1{m_i}+\frac{2d}{m_iR_i}\right)\|\Delta J_i\|_F^2
      +\frac{2d}{m_iR_i}\|\Delta\tilde J_i\|_F^2\right]\\
    &\le\rho_0\|J(X^{(1)})-J(X^{(2)})\|_F^2
      +\rho_J\|\tilde J^{(1)}-\tilde J^{(2)}\|_F^2,
\end{align*}
where the final inequality uses the maxima in \eqref{eq:sampling_constants}.
\end{proof}

\begin{proof}[Proof of Lemma~\ref{lem:jacobian_update_diff}]
Set $\Delta J=J(X^{(1)})-J(X^{(2)})$ and $\Delta\tilde J=\tilde J^{(1)}-\tilde J^{(2)}$. By \eqref{eq:exact_reference_table},
\[
    (\tilde J^{(1)})^+-(\tilde J^{(2)})^+
    =(\mathcal I_J-\Pi_\xi)\Delta\tilde J+\Pi_\xi\Delta J.
\]
Self-adjointness and $\Pi_\xi^2=\Pi_\xi$ give the pathwise cross-term identity
\[
    \langle(\mathcal I_J-\Pi_\xi)\Delta\tilde J,\Pi_\xi\Delta J\rangle_F
    =\langle\Delta\tilde J,(\Pi_\xi-\Pi_\xi^2)\Delta J\rangle_F=0.
\]
Taking total expectation in \eqref{eq:expected_sampling_projection} gives
\[
    \mathbb E\Pi_\xi
    =\mathbb E\!\left[\mathbb E[\Pi^k\mid\mathcal F_k]\right]
    =\sigma\mathcal I_J.
\]
Expanding the square and using this identity gives
\begin{align*}
    \mathbb E\|(\tilde J^{(1)})^+-(\tilde J^{(2)})^+\|_F^2
    &=\mathbb E\|(\mathcal I_J-\Pi_\xi)\Delta\tilde J\|_F^2
      +\mathbb E\|\Pi_\xi\Delta J\|_F^2\\
    &=\langle\Delta\tilde J,(\mathcal I_J-\mathbb E\Pi_\xi)\Delta\tilde J\rangle_F
      +\langle\Delta J,\mathbb E\Pi_\xi\Delta J\rangle_F\\
    &=(1-\sigma)\|\Delta\tilde J\|_F^2+\sigma\|\Delta J\|_F^2.
\end{align*}
\end{proof}

\begin{proof}[Proof of Lemma~\ref{lem:gradient_monotonicity_bound}]
By Assumption~\ref{assump:basic_problem}, each local average $f_i$ is $\mu$-strongly convex. Strong monotonicity of its gradient gives
\[
    \left\langle \Delta X,\Delta\bar G\right\rangle_F
    =\sum_{i=1}^n\left\langle x_i^{(1)}-x_i^{(2)},
      \nabla f_i(x_i^{(1)})-\nabla f_i(x_i^{(2)})\right\rangle
    \ge\mu\sum_{i=1}^n\|x_i^{(1)}-x_i^{(2)}\|^2
    =\mu\|\Delta X\|_F^2.
\]
Convexity and $L$-smoothness of each component imply cocoercivity:
\[
    \left\langle x_i^{(1)}-x_i^{(2)},
    \nabla f_{i,j}(x_i^{(1)})-\nabla f_{i,j}(x_i^{(2)})\right\rangle
    \ge
    \frac{1}{L}\|\nabla f_{i,j}(x_i^{(1)})
    -\nabla f_{i,j}(x_i^{(2)})\|^2 .
\]
Hence
\[
\begin{aligned}
    \left\langle \Delta X,\Delta\bar G\right\rangle_F
    &=
    \sum_{i=1}^n\frac{1}{m_i}\sum_{j=1}^{m_i}
    \left\langle x_i^{(1)}-x_i^{(2)},
    \nabla f_{i,j}(x_i^{(1)})-\nabla f_{i,j}(x_i^{(2)})\right\rangle  \\
    &\ge
    \sum_{i=1}^n\frac{1}{m_iL}\sum_{j=1}^{m_i}
    \|\nabla f_{i,j}(x_i^{(1)})-\nabla f_{i,j}(x_i^{(2)})\|^2
    \ge
    \frac{1}{m_{\max}L}\|\Delta J\|_F^2 .
\end{aligned}
\]
The last inequality uses $m_i\le m_{\max}$ and the definition of $\Delta J$. Averaging the two lower bounds gives
\[
    \langle\Delta X,\Delta\bar G\rangle_F
    =\tfrac12\langle\Delta X,\Delta\bar G\rangle_F
      +\tfrac12\langle\Delta X,\Delta\bar G\rangle_F
    \ge\frac\mu2\|\Delta X\|_F^2
      +\frac1{2m_{\max}L}\|\Delta J\|_F^2.
\]
\end{proof}

\subsection{Finite-Difference Error Estimates}
\label{appendix:finite_difference_estimates}

\begin{proof}[Proof of Lemma~\ref{lem:forward_difference_error}]
The fundamental theorem of calculus gives
\begin{equation}\label{eq:forward_difference_integral}
    \widehat\nabla_u \varphi(x;\beta)
    =
    \frac{\varphi(x+\beta u)-\varphi(x)}{\beta}u
    =
    \left(\int_{0}^{1}\langle\nabla\varphi(x+t\beta u),u\rangle\,dt\right)u
    =\int_{0}^{1}uu^\top\nabla \varphi(x+t\beta u)\,dt.
\end{equation}
Subtracting $uu^\top\nabla\varphi(x)$ from \eqref{eq:forward_difference_integral} and using $L$-smoothness gives
\[
    \|\widehat\nabla_u \varphi(x;\beta)-uu^\top\nabla \varphi(x)\|
    \le\int_0^1\|uu^\top(\nabla\varphi(x+t\beta u)-\nabla\varphi(x))\|\,dt
    \le\int_{0}^{1}Lt\beta\,dt
    =
    \frac{L\beta}{2}.
\]
For any unit vector $v$, \eqref{eq:forward_difference_integral} and $L$-smoothness imply
\begin{align*}
    &\phantom{\leq}\left|\left\langle
      \widehat\nabla_u\varphi(x;\beta)-uu^\top\nabla\varphi(x),v
      \right\rangle\right| \\
    &\le |\langle u,v\rangle|
      \int_0^1\|\nabla\varphi(x+t\beta u)-\nabla\varphi(x)\|\,dt
    \le \frac{L\beta}{2}|\langle u,v\rangle|.
\end{align*}

If $u$ is uniform on $\mathbb S_d$, then $\mathbb E_u\langle u,v\rangle^2=1/d$. Thus, Euclidean norm duality and Cauchy--Schwarz give
\begin{align*}
    &\phantom{\leq}\left\|d\,\mathbb E_u\!\left[
      \widehat\nabla_u\varphi(x;\beta)-uu^\top\nabla\varphi(x)
      \right]\right\|^2 \\
    &=d^2\sup_{\|v\|=1}\left|\mathbb E_u\!\left\langle
      \widehat\nabla_u\varphi(x;\beta)-uu^\top\nabla\varphi(x),v
      \right\rangle\right|^2\\
    &\le \frac{d^2L^2\beta^2}{4}
      \sup_{\|v\|=1}\left(\mathbb E_u|\langle u,v\rangle|\right)^2
    \le \frac{d^2L^2\beta^2}{4}
      \sup_{\|v\|=1}\mathbb E_u\langle u,v\rangle^2
    =\frac{dL^2\beta^2}{4}.
\end{align*}
\end{proof}

\begin{proof}[Proof of Lemma~\ref{lem:uniform_perturbation}]
Write $U=[Z^\top,Y^\top,\tilde J^\top]^\top$ and $X=\operatorname{prox}_{\gamma\Psi}(AZ)=[x_1,\ldots,x_n]^\top$. For $i\in[n]$, $j\in[m_i]$, and $\|u\|=1$, define the observation error at agent $i$ by
\[
    \varepsilon_{i,j,u}
    \coloneqq\widehat\nabla_u f_{i,j}(x_i;\beta)
      -uu^\top\nabla f_{i,j}(x_i).
\]
Lemma~\ref{lem:forward_difference_error} applies at every $x_i$, so
\begin{equation}\label{eq:uniform_observation_error}
    \|\varepsilon_{i,j,u}\|\le\frac{L\beta}{2},
    \qquad
    \|d\,\mathbb E_u\varepsilon_{i,j,u}\|^2\le\frac{dL^2\beta^2}{4}.
\end{equation}
The two operators use the same input and samples. Subtracting their gradient and table updates therefore cancels all stored-table terms and leaves the error blocks
\begin{align*}
    E_\xi^G&\coloneqq
      \widehat G_{\xi}^{\beta}(X,\tilde J)-\widehat G_\xi^0(X,\tilde J),
    & [E_\xi^G]_{i,:}^\top
      &=\frac d{R_i}\sum_{(j,u)\in\mathcal R_i}\varepsilon_{i,j,u},\\
    E_\xi^J&\coloneqq
      \mathcal J_{\xi}^{\beta}(X,\tilde J)-\mathcal J_\xi^0(X,\tilde J),
    & [E_{\xi,i}^J]_{j,:}^\top
      &=\sum_{u:(j,u)\in\mathcal S_i}\varepsilon_{i,j,u}.
\end{align*}
Here $E_{\xi,i}^J$ is the block of $E_\xi^J$ corresponding to agent $i$. For a uniform component $j$ and an independent uniform direction $u$,
set $\bar\varepsilon_i=\mathbb E_{j,u}[d\varepsilon_{i,j,u}]$. Jensen's inequality over $j$ and \eqref{eq:uniform_observation_error} give
\[
    \|\bar\varepsilon_i\|^2
    =\left\|\frac1{m_i}\sum_{j=1}^{m_i}d\,\mathbb E_u\varepsilon_{i,j,u}\right\|^2
    \le\frac1{m_i}\sum_{j=1}^{m_i}
      \|d\,\mathbb E_u\varepsilon_{i,j,u}\|^2
    \le\frac{dL^2\beta^2}{4}.
\]
Applying \eqref{eq:iid_batch_second_moment} with $V_r=d\varepsilon_{i,j_{i,r},u_{i,r}}$ and $a=\bar\varepsilon_i$ gives
\begin{align*}
    \mathbb E\|[E_\xi^G]_{i,:}^\top\|^2
    &=\frac1{R_i}\mathbb E_{j,u}\|d\varepsilon_{i,j,u}\|^2
      +\left(1-\frac1{R_i}\right)\|\bar\varepsilon_i\|^2\\
    &\le\frac1{R_i}\mathbb E_{j,u}\|d\varepsilon_{i,j,u}\|^2
      +\|\bar\varepsilon_i\|^2
    \le\frac{L^2\beta^2}{4}\left(\frac{d^2}{R_i}+d\right).
\end{align*}
The first inequality drops $-\|\bar\varepsilon_i\|^2/R_i$; the last uses the preceding mean bound and \eqref{eq:uniform_observation_error}.
For distinct directions $u,v$ in the same row of $\mathcal S_i$, $\varepsilon_{i,j,u}$ is parallel to $u$, and rowwise orthonormality gives
\[
    \langle\varepsilon_{i,j,u},\varepsilon_{i,j,v}\rangle
    =(u^\top\varepsilon_{i,j,u})(v^\top\varepsilon_{i,j,v})u^\top v=0.
\]
Expanding each row square and using this orthogonality gives
\[
    \|E_\xi^J\|_F^2
    =\sum_{i=1}^n\sum_{j=1}^{m_i}
      \sum_{u:(j,u)\in\mathcal S_i}\|\varepsilon_{i,j,u}\|^2
    \le\frac{L^2\beta^2}{4}\sum_{i=1}^n|\mathcal S_i|.
\]
Taking expectations and using \eqref{eq:sampling_constants} gives
\[
    \mathbb E\|E_\xi^J\|_F^2
    \le\frac{L^2\beta^2}{4}\mathbb E
      \sum_{i=1}^n|\mathcal S_i|
    =\frac{L^2\beta^2\nu}{4}.
\]
Finally, the communication operator maps the update difference to
\[
    \mathbf T_{\xi}^{\beta}U-\mathbf T_\xi^0U
    =\begin{bmatrix}-\gamma E_\xi^G\\-\gamma CE_\xi^G\\E_\xi^J\end{bmatrix}.
\]
Using $A+C^2=I_n$ and the two error estimates above gives
\begin{align*}
    \mathbb E\|\mathbf T_{\xi}^{\beta}U-\mathbf T_\xi^0U\|_{\mathcal M}^2
    &=\gamma^2\mathbb E\|E_\xi^G\|_A^2
      +\gamma^2\mathbb E\|CE_\xi^G\|_F^2
      +c\mathbb E\|E_\xi^J\|_F^2\\
    &=\gamma^2\mathbb E\|E_\xi^G\|_F^2
      +c\mathbb E\|E_\xi^J\|_F^2\\
    &\le\frac{L^2\beta^2}{4}
      \left[\gamma^2\sum_{i=1}^n\left(\frac{d^2}{R_i}+d\right)+c\nu\right]
      =\mathcal E_\beta.
\end{align*}
\end{proof}

For the prescribed $c$, the sampling constants give
\begin{equation}\label{eq:sampling_sum_bound}
    \sum_{i=1}^n\frac{d^2}{R_i}
    =\sum_{i=1}^n\frac{m_i d}{2}\frac{2d}{m_iR_i}
    \le\frac{Md\rho_J}{2},
    \qquad c\nu=4\gamma^2Md\rho_J.
\end{equation}
The inequality uses the definition of $\rho_J$, and the last equality uses $\nu=\sigma Md$. Substitution into \eqref{eq:uniform_perturbation_bound} yields
\begin{equation}\label{eq:simplified_perturbation_bound}
    \mathcal E_\beta
    \le\frac{\gamma^2dL^2\beta^2}{4}
      \left(n+\frac{9M\rho_J}{2}\right)
    \le\frac54\gamma^2dL^2(n+M\rho_J)\beta^2.
\end{equation}
The last inequality uses $1\le5$ and $9/2\le5$.

\subsection{Chebyshev Communication Properties}
\label{appendix:chebyshev_properties}

\begin{proof}[Proof of Lemma~\ref{lem:chebyshev_spectral_gap}]
Suppose first that $0<\rho<1$. The recurrence defining $T_r$ and the cosine and hyperbolic-cosine addition formulas give, by induction on $r$,
\[
    T_r(\cos t)=\cos(rt),\qquad
    T_r(\cosh t)=\cosh(rt),\qquad t\in\mathbb R,\quad r\ge0.
\]
Set $\alpha\coloneqq\operatorname{arcosh}(1/\rho)>0$. For $K\ge1$, these identities imply
\begin{equation}\label{eq:chebyshev_polynomial_bound}
    T_K(1/\rho)=\cosh(K\alpha)>1,\qquad
    P_K(1)=1,\qquad
    \max_{|t|\le\rho}|P_K(t)|
    =\frac{1}{\cosh(K\alpha)}<1.
\end{equation}
The maximum uses $|T_K(s)|\le1$ for $s\in[-1,1]$, with equality at $s=1$. Let $Q\coloneqq\mathbf1\mathbf1^\top/n$.
Since $W$ is symmetric, $W\mathbf1=\mathbf1$, and $\|W-Q\|_2\le\rho$, its restriction to $\mathbf1^\perp$ has all eigenvalues in $[-\rho,\rho]$.
Polynomial functional calculus and \eqref{eq:chebyshev_polynomial_bound} give
\[
    P_K(W)\mathbf1=\mathbf1,\qquad
    \left\|P_K(W)-Q\right\|_2
    \le\frac{1}{\cosh(K\alpha)}<1.
\]
Thus $A_K$ has eigenvalue $1$ on the consensus subspace and eigenvalues strictly between $0$ and $1$ on its orthogonal complement.
The matrix $B_K=I_n-A_K$ has eigenvalue $0$ exactly on the consensus subspace, and
\[
    \lambda_2(B_K)
    \ge\frac{1-\cosh(K\alpha)^{-1}}{2}.
\]
This proves \eqref{eq:chebyshev_operator_properties} and the first inequality in \eqref{eq:chebyshev_gap_bound}. Moreover, $\cosh\alpha=1/\rho$ and $\tanh\alpha\le\alpha$ yield
\[
    \alpha^2\ge\tanh^2\alpha=1-\rho^2,\qquad
    \cosh(K\alpha)\ge1+\frac{K^2\alpha^2}{2}
    \ge1+\frac{K^2(1-\rho^2)}{2}.
\]
Substituting this lower bound into the preceding gap estimate gives
\[
    \lambda_2(B_K)
    \ge\frac12\left(1-
       \frac{1}{1+K^2(1-\rho^2)/2}\right)
    =\frac{K^2(1-\rho^2)}{4+2K^2(1-\rho^2)}.
\]
Since $0<\tau\le\sigma/4\le1/4$, the choice \eqref{eq:chebyshev_round_choice} gives $K^2(1-\rho^2)\ge8\tau$. The function $s\mapsto s/(4+2s)$ is increasing for $s\ge0$, so
\[
    \lambda_2(B_K)
    \ge\frac{8\tau}{4+16\tau}
    =\frac{2\tau}{1+4\tau}\ge\tau.
\]
The last inequality uses $\tau\le1/4$.

To verify the implementation, stack the vectors in Algorithm~\ref{alg:chebyshev_consensus} as $\chi^{k,r}\coloneqq[\chi_1^{k,r},\ldots,\chi_n^{k,r}]^\top\in\mathbb R^{n\times d}$.
Its scalar recurrence and initialization give $a_r=T_r(1/\rho)>0$. For $r=0,1$,
\[
    \chi^{k,0}=\chi^k=\frac{T_0(W/\rho)}{a_0}\chi^k,
    \qquad
    \chi^{k,1}=W\chi^k=\frac{T_1(W/\rho)}{a_1}\chi^k.
\]
If this representation holds at $r$ and $r-1$, its update gives
\begin{align*}
    \chi^{k,r+1}
    &=\frac{2a_r}{\rho a_{r+1}}W\chi^{k,r}
      -\frac{a_{r-1}}{a_{r+1}}\chi^{k,r-1}\\
    &=\frac{2(W/\rho)T_r(W/\rho)-T_{r-1}(W/\rho)}{a_{r+1}}
      \chi^k
    =\frac{T_{r+1}(W/\rho)}{a_{r+1}}\chi^k.
\end{align*}
The final equality is the Chebyshev recurrence. Induction therefore gives
\[
    \zeta^{k+1}=\frac{\chi^k+\chi^{k,K}}2
    =\frac{I_n+P_K(W)}2\chi^k=A_K\chi^k.
\]
There is one multiplication by $W$ at $r=1$ and one at each of the remaining $K-1$ rounds, each implemented by a neighbor exchange.
If $\rho=0$, then $W=Q$ and one ordinary round gives $A_1=(I_n+Q)/2$ and $B_1=(I_n-Q)/2$. All the stated matrix properties hold, and $\lambda_2(B_1)=1/2\ge\tau$.
\end{proof}

\subsection{Oracle and Communication Complexity}
\label{appendix:complexity_proofs}

\begin{proof}[Proof of Corollaries~\ref{cor:unaccelerated_complexity}
and~\ref{cor:accelerated_complexity}]
Let $\Pi_i^t$ be the restriction of $\Pi^t$ to agent $i$'s $md$-dimensional table space. Both corollaries assume $m_i=m$ and $|\mathcal S_i^t|=R_i=R$ for every agent, giving
\begin{align*}
    R
    &=\mathbb E[|\mathcal S_i^t|\mid\mathcal F_t]
      =\operatorname{tr}\!\left(\mathbb E[\Pi_i^t\mid\mathcal F_t]\right)
      =\sigma md.
\end{align*}
The first equality uses the fixed refresh size; the second uses $\operatorname{tr}(uu^\top)=1$ for each sampled unit direction in \eqref{eq:sampling_projection_definition};
the last uses \eqref{eq:expected_sampling_projection}. Substituting $m_i=m$ and $R_i=R$ into \eqref{eq:sampling_constants} gives
\begin{gather*}
    \sigma=\frac{R}{md},\qquad
    \rho_0=\frac{1+2d/R}{m},\qquad
    \rho_J=\frac{2d}{mR},\\
    mL(\rho_0+4\rho_J)
      =\frac{L(R+10d)}{R}=\frac1\gamma,\qquad
    n+M\rho_J=n\left(1+\frac{2d}{R}\right).
\end{gather*}
Thus the stepsizes in \eqref{eq:unaccelerated_batch_parameters} and~\eqref{eq:accelerated_batch_parameters} are admissible,
and the theorem's $\tau=\min\{\gamma\mu/2,\sigma/4\}$ equals the value in those equations. In particular,
\begin{align*}
    \frac1\tau
    &=\max\left\{\frac{2\kappa(R+10d)}{R},
      \frac{4md}{R}\right\}\\
    &=\Theta\!\left(
      \kappa\left(1+\frac dR\right)
        +\frac{md}{R}\right).
\end{align*}
The order estimate uses $R+d\le R+10d\le10(R+d)$ and $(u+v)/2\le\max\{u,v\}\le u+v$ for $u,v\ge0$.

For the ordinary algorithm, symmetry of $W$ and the definition of $\rho$ give
\[
    \lambda_2(B)=\frac{1-\lambda_{n-1}(W)}2
      \ge\frac{1-\rho}2
      =\frac{1-\rho^2}{2(1+\rho)}
      \ge\frac{1-\rho^2}{4}>0.
\]
The first inequality uses $\lambda_{n-1}(W)\le\rho$ on the subspace orthogonal to $\mathbf1$; the last uses $0\le\rho<1$.
Set $a=(1-\rho^2)/4$ and $q=\min\{\tau,(1-\rho^2)/4\}$. Then
\[
    1-\delta=\min\{\tau,\lambda_2(B)\}\ge q,
    \qquad
    \mathcal D(\lambda_2(B))\le\mathcal D(a).
\]
The second inequality follows from \eqref{eq:initial_error_prefactor}, whose only $a$-dependent term is nonnegative and proportional to $a^{-1}$.
For Chebyshev acceleration, set $(a,q)=(\tau,\tau)$. In both cases, $0<q\le\tau\le1/4$, and the corresponding theorem gives
\[
    \mathbb E\|X^k-X^*\|_F^2
    \le\left(1-\frac q2\right)^{k-1}\mathcal D(a)
      +\frac{5\gamma^2ndL^2(1+2d/R)\beta^2}{q^2}.
\]
For the ordinary algorithm, this bound also uses $1-\delta\ge q$ to bound the contraction factor and error floor, and the preceding bound on the initial-error prefactor.
The choices of $\beta$ and $k$ in the respective corollary satisfy
\begin{align*}
    \frac{5\gamma^2ndL^2(1+2d/R)\beta^2}{q^2}
    &\le\frac{\epsilon}{2},\\
    \frac{q(k-1)}2
    &\ge\log\max\left\{1,\frac{2\mathcal D(a)}{\epsilon}\right\},\\
    e^{-q(k-1)/2}\mathcal D(a)
    &\le\frac{\mathcal D(a)}{\max\{1,2\mathcal D(a)/\epsilon\}}
      \le\frac{\epsilon}{2}.
\end{align*}
Using $1-q/2\le e^{-q/2}$ in the rate therefore gives
\[
    \mathbb E\|X^k-X^*\|_F^2
    \le e^{-q(k-1)/2}\mathcal D(a)
      +\frac{5\gamma^2ndL^2(1+2d/R)\beta^2}{q^2}
    \le\frac{\epsilon}{2}+\frac{\epsilon}{2}
    =\epsilon.
\]

Each iteration uses $R_i+|\mathcal S_i^t|=2R$ oracle units per agent. Computing $X^k$ requires the updates indexed by $t=0,\ldots,k-1$, so the total oracle cost per agent is
\[
    \sum_{t=0}^{k-1}(R_i+|\mathcal S_i^t|)=2R k.
\]
Each oracle unit requires at most two function values by \eqref{eq:fm_2p}. The prescribed iteration count satisfies
\[
    k\le2+\frac2q\log\max\left\{1,\frac{2\mathcal D(a)}{\epsilon}\right\}
    =\mathcal O\!\left(\frac1q
      \log\left(2+\frac{\mathcal D(a)}{\epsilon}\right)\right).
\]
The inequality uses $\lceil z\rceil\le z+1$; the asymptotic bound uses $q\le1/4$ and $\log(2+\mathcal D(a)/\epsilon)\ge\log2$. For the ordinary algorithm, there is one communication round per update.
Substituting $a=(1-\rho^2)/4$, $q=\min\{\tau,(1-\rho^2)/4\}$, and the estimate on $\tau^{-1}$ therefore gives
\begin{align*}
    2R k
    &=\mathcal O\!\left(
      \left[\kappa(R+d)+md
        +\frac{R}{1-\rho^2}\right]
      \log\left(2+\frac{\mathcal D((1-\rho^2)/4)}{\epsilon}\right)\right),\\
    k&=\mathcal O\!\left(
      \left[\kappa\left(1+\frac dR\right)
        +\frac{md}{R}+\frac1{1-\rho^2}\right]
      \log\left(2+\frac{\mathcal D((1-\rho^2)/4)}{\epsilon}\right)\right).
\end{align*}
Here the same maximum--sum comparison bounds $q^{-1}=\max\{\tau^{-1},4/(1-\rho^2)\}$.
When $W\succeq0$, $\rho=\lambda_{n-1}(W)$, so
\[
    \lambda_2(B)=\frac{1-\rho}2
      =\frac{1-\rho^2}{2(1+\rho)}
      =\Theta(1-\rho^2).
\]
Thus the substitution preserves the spectral-gap order for positive-semidefinite mixing matrices.
For any admissible $W$, lazy mixing replaces it by $(I_n+W)/2\succ0$ and replaces $B$ by $B/2$, so this order comparison is available without extra communication and with at most a constant-factor change in the unaccelerated complexity.

For Chebyshev acceleration, each update uses $K$ communication rounds, and \eqref{eq:chebyshev_round_choice} implies
\[
    \frac K\tau\le\frac2\tau+
      \sqrt{\frac{8}{\tau(1-\rho^2)}}.
\]
This uses $\lceil\max\{1,z\}\rceil\le2+z$ for $z\ge0$ and also holds when $\rho=0$, $K=1$. Substituting $(a,q)=(\tau,\tau)$ and the estimate on $\tau^{-1}$ gives
\[
    2Rk=\mathcal O\!\left(
      \left[\kappa(R+d)+md\right]
      \log\left(2+\frac{\mathcal D(\tau)}{\epsilon}\right)\right)
\]
and
\[
    kK=\mathcal O\!\left(\left[
      \kappa\left(1+\frac dR\right)+\frac{md}{R}
      +\sqrt{\frac{\kappa(1+d/R)+md/R}{1-\rho^2}}\right]
      \log\left(2+\frac{\mathcal D(\tau)}{\epsilon}\right)\right).
\]
\end{proof}

\end{document}